\documentclass[reqno, 11pt]{amsart}
\usepackage{amsthm, amsmath, amssymb, graphicx, tikz,enumerate,booktabs}
\usepackage[left=1in,right=1in,top=1in,bottom=1in]{geometry}

\usepackage[colorlinks=false, pdfstartview=FitV, 
pagebackref=false,hidelinks]{hyperref}

\usepackage[color,final]{showkeys} 

\colorlet{refkey}{orange!20}
\colorlet{labelkey}{blue!30}

\usetikzlibrary{calc,intersections,through,backgrounds,shapes,patterns}
\tikzstyle{p}+=[fill=black, circle, minimum width = 1pt, inner sep =
1pt]

\newtheorem{theorem}{Theorem}[section]
\newtheorem{proposition}[theorem]{Proposition}
\newtheorem{lemma}[theorem]{Lemma}

\theoremstyle{definition}
\newtheorem{definition}[theorem]{Definition}

\theoremstyle{remark}
\newtheorem*{remark}{Remark}

\newtheoremstyle{named}{}{}{\itshape}{}{\bfseries}{.}{.5em}{#3}
\theoremstyle{named}
\newtheorem*{namedtheorem}{}

\newcommand{\abs}[1]{\left\lvert#1\right\rvert}
\newcommand{\sabs}[1]{\lvert#1\rvert}
\newcommand{\norm}[1]{\left\lVert#1\right\rVert}
\newcommand{\snorm}[1]{\lVert#1\rVert}
\newcommand{\ang}[1]{\left\langle #1 \right\rangle}
\newcommand{\sang}[1]{\langle #1 \rangle}

\newcommand{\paren}[1]{\left( #1 \right)}

\newcommand{\wt}{\widetilde}

\renewcommand{\Re}{\operatorname{Re}}

\newcommand{\x}{\times}
\newcommand{\e}{\epsilon}

\newcommand{\EE}{\mathbb{E}}
\newcommand{\RR}{\mathbb{R}}

\newcommand{\NN}{\mathbb{N}}
\newcommand{\ZZ}{\mathbb{Z}}

\newcommand{\tg}{\tilde{g}}
\newcommand{\tf}{\tilde{f}}
\newcommand{\tA}{\widetilde{A}}

\newcommand{\bx}{\mathbf{x}}
\newcommand{\by}{\mathbf{y}}

\title{The Green--Tao theorem: an exposition}

\author{David Conlon}
\address{Mathematical Institute, Oxford OX2 6GG, United Kingdom}
\email{david.conlon@maths.ox.ac.uk}
\thanks{The first author was supported by a Royal Society University
  Research Fellowship.}

\author{Jacob Fox}
\address{Department of Mathematics, MIT, Cambridge, MA 02139-4307}
\email{fox@math.mit.edu}
\thanks{The second author was supported by a Packard
  Fellowship, NSF Career Award DMS-1352121, an Alfred P. Sloan Fellowship, and an
  MIT NEC Corporation Award.}

\author{Yufei Zhao}
\address{Department of Mathematics, MIT, Cambridge, MA 02139-4307}
\email{yufeiz@math.mit.edu}
\thanks{The third author was supported by a
  Microsoft Research PhD Fellowship.}

\begin{document}

\begin{abstract}
  The celebrated Green-Tao theorem states that the prime numbers
  contain arbitrarily long arithmetic progressions. We give an
  exposition of the proof, incorporating several simplifications that
  have been discovered since the original paper.
\end{abstract}

\maketitle

\setcounter{tocdepth}{1}

\section{Introduction} \label{sec:intro}

In 2004, Ben Green and Terence Tao \cite{GT08} proved the following celebrated theorem, resolving a folklore conjecture about prime numbers.

\begin{theorem}[Green-Tao] \label{thm:gt}
The prime numbers contain arbitrarily long arithmetic progressions.
\end{theorem}

Our intention is to give a complete proof of this theorem.
Although there have been numerous other expositions \cite{Gow10,G07,Host06,Kra06,Tao06col,Tao06pamq,Tao07icm},
we were prompted to write this note because of our
recent work \cite{CFZrelsz,Zhao14} simplifying one of the key technical
ingredients in the proof. Together with work of Gowers \cite{Gow10}, Reingold, Trevisan, Tulisiani, and Vadhan \cite{RTTV08}, and Tao \cite{Taonote},
there have now been substantial simplifications to almost every aspect of the proof. We have chosen to collect these simplifications and present an
up-to-date exposition in order to make the proof more accessible.

A key element in the proof of Theorem~\ref{thm:gt} is Szemer\'edi's theorem~\cite{Sze75} on arithmetic progressions in dense subsets of the integers. To state this theorem, we define the \emph{upper
  density} of a set $A \subseteq \NN$ to be
\[
\limsup_{N \to \infty} \frac{\abs{A \cap [N]}}{N}, \qquad
\text{where } [N] := \{1, 2, \dots, N\}.
\]
\begin{theorem} [Szemer\'edi] \label{thm:sz}
  Every subset of $\NN$ with positive upper density contains
  arbitrarily long arithmetic progressions.
\end{theorem}

Szemer\'edi's theorem is a deep and important result and the original
proof~\cite{Sze75} is long and complex. It has had a huge impact on the subsequent development
of combinatorics and, in particular, was responsible for the introduction
of the \emph{regularity lemma}, now a cornerstone of
modern combinatorics. Numerous different proofs of Szemer\'edi's theorem have since been
discovered and all of them have introduced important new ideas that grew
into active areas of research. The three main modern approaches to
Szemer\'edi's theorem are by ergodic theory~\cite{F77,FKO82}, higher
order Fourier analysis~\cite{Gow98,Gow01}, and hypergraph
regularity~\cite{Gow07, NRS06, RS04, RS06, Tao06jcta}. However, none
of these approaches are easy. We shall therefore assume Szemer\'edi's theorem as
a black box and explain how to derive the Green-Tao theorem using it.

As the set of primes has density zero, Szemer\'edi's theorem does not
immediately imply the Green-Tao theorem. Nevertheless, Erd\H{o}s famously
conjectured that the density of the primes alone should guarantee the existence
of long APs.\footnote{For brevity, we will usually write AP for arithmetic
progression and $k$-AP for a $k$-term AP.} Specifically, he conjectured
that any subset $A$ of $\NN$ with divergent harmonic sum, i.e., $\sum_{a \in A} 1/a = \infty$,
must contain arbitrarily long APs. This conjecture is widely believed to be
true, but it has yet to be proved even in the case of 3-term
APs.\footnote{A recent result of Sanders \cite{San11} is within
  a hair's breadth of verifying Erd\H{o}s' conjecture for
  3-APs. Sanders proved that every $3$-AP-free subset of $[N]$ has size at most $O(N (\log\log N)^6 / \log N)$, which is just slightly shy of the logarithmic density barrier that one wishes to
  cross (see Bloom~\cite{Bloom} for a recent improvement). In the other direction, Behrend \cite{Be46} constructed a $3$-AP-free subset of $[N]$ of size $Ne^{-O(\sqrt{\log N})}$. There is some evidence to suggest that Behrend's lower bound is closer to the truth (see \cite{SS14}). For longer APs, the gap is much larger. The best upper bound, due to Gowers \cite{Gow01}, is that every $k$-AP-free subset of $[N]$ has size at most $N/(\log\log N)^{c_k}$ for some
$c_k > 0$ (though for
$k=4$ there have been some improvements \cite{GT09r4}).}

If not by density considerations, how do Green and Tao prove their
theorem? The answer is that they treat Szemer\'edi's theorem as a black box
and show, through a \emph{transference principle}, that a Szemer\'edi-type statement holds
relative to sparse pseudorandom subsets of the integers, where a set is said to be \emph{pseudorandom} if it resembles a random set
of similar density in terms of certain statistics or properties. We refer to such a statement as a
\emph{relative Szemer\'edi theorem}. Given two sets $A$ and $S$ with $A \subseteq S$, we define the
relative upper density of $A$ in $S$ to be $\limsup_{N
  \to \infty} \abs{A \cap [N]}/\abs{S \cap [N]}$.

\begin{namedtheorem}[Relative Szemer\'edi theorem] \emph{(Informally)}
If $S$
is a (sparse) set of
integers satisfying certain pseudorandomness conditions and $A$ is a subset
of $S$ with positive relative density, then $A$ contains arbitrarily long APs.
\end{namedtheorem}

To prove the Green-Tao theorem, it then suffices to show that there is a set of ``almost primes''
containing, but not much larger than, the primes which satisfies the required pseudorandomness
conditions. In the work of Green and Tao, there are two such conditions, known as the linear
forms condition and the correlation condition.

The proof of the Green-Tao theorem therefore falls into two parts, the first part being the proof of the relative
Szemer\'edi theorem and the second part being the construction of an appropriately pseudorandom
superset of the primes. Green and Tao credit the contemporary work of Goldston and Y\i ld\i r\i m~\cite{GY03} for the construction and estimates used in the second half of the proof. Here we will follow a simpler approach discovered by Tao \cite{Taonote}.

The proof of the relative Szemer\'edi theorem also splits into two parts, the dense model theorem and the counting lemma.
Roughly speaking, the dense model theorem allows us to say that if $S$ is a sufficiently pseudorandom set then any relatively dense subset $A$ of $S$
may be ``approximated'' by a dense subset $\tA$ of $\mathbb{N}$, while the counting lemma shows
that the number of arithmetic progressions in $A$ is close, up to a normalization factor, to the number of arithmetic progressions
in $\tA$. Since $\tA$ is a dense subset of $\mathbb{N}$, Szemer\'edi's theorem implies that $\tA$ contains arbitrarily long
APs and this in turn implies that $A$ contains arbitrarily long APs.

This is also the outline we will follow in this paper, though for each part we will follow a different approach to the original paper. For the counting lemma, we will follow the recent approach taken by the authors in \cite{CFZrelsz}. This approach has significant advantages over the original method of Green and Tao, not least of which is that a weakening of the linear forms condition is sufficient for the relative Szemer\'edi theorem to hold. This means that the estimates involved in verifying the correlation condition may now be omitted from the proof.

In \cite{CFZrelsz}, the dense model theorem was replaced with a certain sparse regularity lemma. However, as subsequently observed by Zhao \cite{Zhao14}, the original dense model theorem may also be used. To prove the dense model theorem, we will follow an elegant method developed independently by Gowers \cite{Gow10} and by Reingold, Trevisan, Tulsiani, and Vadhan \cite{RTTV08}.

The $3$-AP case of Szemer\'edi's theorem was first proved by Roth~\cite{Roth53} in the 1950s.
While Roth's theorem, as this case is usually known, is already a very interesting and
nontrivial result, the $3$-AP case is substantially easier than the general result. In contrast, when proving a relative Szemer\'edi
theorem by \emph{transferring} Szemer\'edi's theorem down to the sparse setting, the general case is not mathematically more difficult than the $3$-AP case. However, as one might expect, the
notation for the general case can be rather cumbersome. For this reason, we explain various aspects of the proof first for 3-APs and only afterwards discuss how it can be adapted to the general case.

We begin, in Section \ref{sec:roth}, by presenting the Ruzsa-Szemer\'edi  graph-theoretic approach to Roth's theorem. In particular, we present a graph-theoretic construction that will motivate the definition of the linear forms conditions, which we state in Sections \ref{sec:rel-roth} and \ref{sec:rel-sz}, first for Roth's theorem, then for Szemer\'edi's theorem. The dense model theorem and the counting lemma are explained in
Sections \ref{sec:dense-model} and \ref{sec:count}, respectively. We conclude the proof of the relative Szemer\'edi theorem in Section \ref{sec:pf-rel-sz}.
In Sections \ref{sec:maj} and \ref{sec:verify-lfc}, we will construct the relevant set of almost primes (or rather a majorizing measure for the primes) and show that it satisfies the linear forms condition. We conclude with some remarks about extensions of the Green-Tao theorem.

\section{Roth's theorem via graph theory} \label{sec:roth}

One way to state Szemer\'edi's theorem is that for every fixed $k$ every $k$-AP-free subset of $[N]$ has $o(N)$ elements. It is not hard to prove that this ``finitary'' version of
Szemer\'edi's theorem is equivalent to the ``infinitary'' version
stated as Theorem~\ref{thm:sz}.

In fact, it will be more convenient to work in the setting of the
abelian group $\ZZ_N := \ZZ/N\ZZ$ as opposed to $[N]$. These two
settings are roughly equivalent for studying $k$-APs, with the only difference being that
$\ZZ_N$ allows APs to wrap around $0$. For example, $N-1, 0, 1$ is a $3$-AP in $\ZZ_N$, but not in $[N]$. To deal with this issue,
one simply embeds $[N]$ into a slightly larger cyclic group
so that no $k$-APs wrap around zero. Working in $\ZZ_N$, we will now show how Roth's theorem follows from a result in graph theory.

\begin{theorem} [Roth] \label{thm:roth}
  If $A \subseteq \ZZ_N$ is $3$-AP-free, then $\abs{A} = o(N)$.
\end{theorem}

\begin{figure}
  \centering
  \begin{tikzpicture}
    [scale=.6,
    V/.style = {draw, fill=white,circle, inner sep =
      0em, outer sep = .1em, minimum size=5em},
    b/.style={ball color = gray,circle,inner sep = 0,minimum
      size=2mm},
    e/.style={line width=.5mm,line cap = rect}
    ]
    \begin{scope}[shift={(10,0)}]
      \node at (-4,3.5) {$G_A$};

      \path[use as bounding box] (-5,-3) rectangle (5,4.3);
      \fill[blue!20,opacity=.4] ($(90:2.5) + (150:2.5em)$) -- ($(90:2.5) + (150:-2.5em)$) --
      ($(210:2.5) + (150:-2.5em)$) -- ($(210:2.5)+(150:2.5em)$);
      \fill[blue!20,opacity=.4] ($(90:2.5) + (30:2.5em)$) -- ($(90:2.5) + (30:-2.5em)$) --
      ($(330:2.5) + (30:-2.5em)$) -- ($(330:2.5)+(30:2.5em)$);
      \fill[blue!20,opacity=.4] ($(330:2.5) + (90:2.5em)$) -- ($(330:2.5) + (90:-2.5em)$) --
      ($(210:2.5) + (90:-2.5em)$) -- ($(210:2.5)+(90:2.5em)$);
      \draw[circle,fill=white] (90:2.5) circle (2.5em) {};
      \draw[circle,fill=white] (210:2.5) circle (2.5em) {};
      \draw[circle,fill=white] (330:2.5) circle (2.5em) {};

      \node at (90:4) {$X = \ZZ_N$};
      \node at (200:4.7) {$Y = \ZZ_N$};
      \node at (340:4.7) {$Z = \ZZ_N$};

      \node[b,label=above:{\small $x$}] (x) at (80:2){};
      \node[b,label=left:{\small $y$}] (y) at (220:2){};
      \node[b,label=right:{\small $z$}] (z) at (330:2){};

      \draw[e] (x)--(y);
      \node[font=\footnotesize,align=center] at (150:2.8)
      {$x \sim y$ iff \\ $2x + y \in A$};

      \draw[e] (x)--(z);
      \node[font=\footnotesize,align=center] at (30:2.8)
      {$x \sim z$ iff \\ $x-z \in A$};

      \draw[e] (y)--(z);
      \node[font=\footnotesize,align=center] at (-90:2.4)
      {$y \sim z$ iff \\ $-y-2z \in A$};
    \end{scope}
  \end{tikzpicture}
  \caption{The construction in the proof of Roth's theorem.} \label{fig:G_A}
\end{figure}
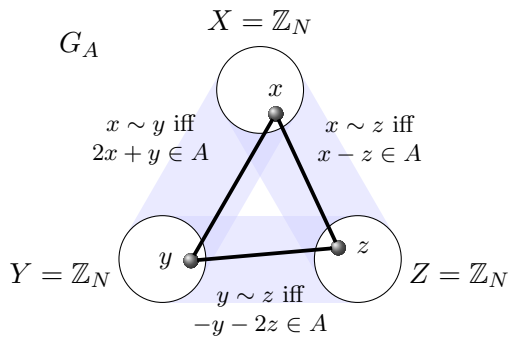

Consider the following graph construction (see
Figure~\ref{fig:G_A}). Given $A \subseteq \ZZ_N$, we construct a
tripartite graph $G_A$ whose vertex sets are $X$, $Y$, and $Z$, each with $N$
vertices labeled by elements of $\ZZ_N$. The edges are constructed as
follows (one may think of this as a variant of the Cayley graph for
$\ZZ_N$ generated by $A$):
\begin{itemize}
\item $(x,y) \in X \x Y$ is an edge if and only if $2x + y \in A$;
\item $(x,z) \in X \x Z$ is an edge if and only if $x - z \in A$;
\item $(y,z) \in Y \x Z$ is an edge if and only if $-y-2z \in A$.
\end{itemize}
Observe that $(x,y,z) \in X \x Y \x Z$ forms a triangle if and only if
all three of
\[
2x+y,\quad x-z, \quad -y-2z
\]
are in $A$. These numbers form a $3$-AP with common difference $-x-y-z$, so we see that
triangles in $G_A$ correspond to $3$-APs in $A$.

However, we assumed that $A$ is $3$-AP-free. Does this mean that $G_A$
has no triangles? Not quite. There are still some triangles
in $G_A$, namely those that correspond to trivial $3$-APs in $A$, i.e.,
$3$-APs with common difference zero. So the triangles in $G_A$ are
precisely those with $x + y + z = 0$. This easily implies
that every edge in $G_A$ is
contained in exactly one triangle, namely the one that completes
the equation $x+y+z = 0$.

What can we say about a graph where every edge is contained in exactly
one triangle? The following result of Ruzsa and Szemer\'edi \cite{RS78} shows that it cannot have many edges.

\begin{theorem}[Ruzsa-Szemer\'edi] \label{thm:RS}
 If $G$ is a graph on $n$ vertices with every edge in exactly one triangle, then $G$ has $o(n^2)$ edges.
\end{theorem}

Our graph $G_A$ has $3N$ vertices and $3N|A|$ edges (for
every $x \in X$, there are exactly $\abs{A}$ vertices $y \in Y$ with
$2x + y \in A$ and similarly for $Y \x Z$ and $X \x Z$). So it follows by
Theorem~\ref{thm:RS} that $3N |A| = o((3N)^2)$. Hence $\abs{A}
= o(N)$, proving Roth's theorem.

Theorem~\ref{thm:RS} easily follows from a result known as the \emph{triangle removal lemma},
which says that if a graph on $n$ vertices has $o(n^3)$ triangles,
then it can be made triangle-free by removing $o(n^2)$ edges.
Though both results look rather innocent, it is only recently \cite{CF13,F11} that a proof was found
which avoids the use of Szemer\'edi's regularity lemma.

We will not include a proof of Theorem~\ref{thm:RS} here, since this
would lead us too far down the route of proving Szemer\'edi's theorem. However, if our purpose was
not to prove Roth's theorem, then why translate it into graph-theoretic language in the first place? The reason is that
the counting lemma and pseudorandomness conditions used for transferring Roth's theorem
to the sparse setting are most naturally phrased in terms of graph theory.\footnote{However, it is worth stressing that the bounds in the relative Roth theorem do not reflect the poor bounds given by the graph theoretic approach to Roth's theorem. While graph theory is a convenient language for phrasing the transference principle, Roth's theorem itself only appears as a black box and any bounds we have for this theorem transfer directly to the sparse version.} We will begin to make
this explicit in the next section.

\section{Relative Roth theorem} \label{sec:rel-roth}

In this section, we describe the relative Roth theorem. We first give an informal statement.

\begin{namedtheorem}[Relative Roth Theorem] \emph{(Informally)} If $S
  \subseteq \ZZ_N$ satisfies certain pseudorandomness conditions and $A
  \subseteq S$ is $3$-AP-free, then $\abs{A} = o(\abs{S})$.
\end{namedtheorem}

\begin{figure}

  \newcommand{\VR}{2.5cm}
  \newcommand{\Vr}{2.5em}

  \begin{tikzpicture}[scale=.5,
    b/.style={ball color = gray,circle,inner sep = 0,minimum
      size=1mm},
    e/.style={line width=.5mm,line cap = rect}
    ]

    \begin{scope}[shift={(-2,0)}]
      \node at (0,5) {$G_S$};

      \fill[blue!20,opacity=.4] ($(90:2.5) + (150:2.5em)$) -- ($(90:2.5) + (150:-2.5em)$) --
      ($(210:2.5) + (150:-2.5em)$) -- ($(210:2.5)+(150:2.5em)$);
      \fill[blue!20,opacity=.4] ($(90:2.5) + (30:2.5em)$) -- ($(90:2.5) + (30:-2.5em)$) --
      ($(330:2.5) + (30:-2.5em)$) -- ($(330:2.5)+(30:2.5em)$);
      \fill[blue!20,opacity=.4] ($(330:2.5) + (90:2.5em)$) -- ($(330:2.5) + (90:-2.5em)$) --
      ($(210:2.5) + (90:-2.5em)$) -- ($(210:2.5)+(90:2.5em)$);
      \draw[circle,fill=white] (90:2.5) circle (2.5em) {};
      \draw[circle,fill=white] (210:2.5) circle (2.5em) {};
      \draw[circle,fill=white] (330:2.5) circle (2.5em) {};

      \node at (50:4) {$X = \ZZ_N$};
      \node at (200:4.7) {$Y = \ZZ_N$};
      \node at (340:4.7) {$Z = \ZZ_N$};

      \node[b,label=above:{\small $x$}] (x) at (80:2){};
      \node[b,label=left:{\small $y$}] (y) at (220:2){};
      \node[b,label=right:{\small $z$}] (z) at (330:2){};

      \draw[e] (x)--(y);
      \node[font=\footnotesize,align=center] at (155:2.8)
      {$x \sim y$ iff \\ $2x + y \in S$};

      \draw[e] (x)--(z);
      \node[font=\footnotesize,align=center] at (25:3)
      {$x \sim z$ iff \\ $x-z \in S$};

      \draw[e] (y)--(z);
      \node[font=\footnotesize,align=center] at (-90:2.4)
      {$y \sim z$ iff \\ $-y-2z \in S$};
    \end{scope}

    \begin{scope}[shift={(10,0)}]
      \fill[blue!20,opacity=.4] ($(90:\VR) + (150:\Vr)$) -- ($(90:\VR) + (150:-\Vr)$) --
      ($(210:\VR) + (150:-\Vr)$) -- ($(210:\VR)+(150:\Vr)$);
      \fill[blue!20,opacity=.4] ($(90:\VR) + (30:\Vr)$) -- ($(90:\VR) + (30:-\Vr)$) --
      ($(330:\VR) + (30:-\Vr)$) -- ($(330:\VR)+(30:\Vr)$);
      \fill[blue!20,opacity=.4] ($(330:\VR) + (90:\Vr)$) -- ($(330:\VR) + (90:-\Vr)$) --
      ($(210:\VR) + (90:-\Vr)$) -- ($(210:\VR)+(90:\Vr)$);
      \draw[circle,fill=white] (90:\VR) circle (\Vr) {};
      \draw[circle,fill=white] (210:\VR) circle (\Vr) {};
      \draw[circle,fill=white] (330:\VR) circle (\Vr) {};

      \node[b] (x) at (80:2.5){};
      \node[b] (x') at (100:2.5){};
      \node[b] (y) at (200:2.5){};
      \node[b] (y') at (220:2.5){};
      \node[b] (z) at (320:2.5){};
      \node[b] (z') at (340:2.5){};

      \draw[e] (x)--(y)--(x')--(y')--(x);
      \draw[e] (x)--(z)--(x')--(z')--(x);
      \draw[e] (y)--(z)--(y')--(z')--(y);

      \node[align=center,font=\small] at (5,1) {$K_{2,2,2}$  \&
        subgraphs, \\ e.g.,};

      \node[align=center] at (5,4.5) {Pseudorandomness hypothesis for $S
        \subseteq \ZZ_N$:
        \\ {\small $G_S$ has asymptotically the expected number of embeddings of}};
    \end{scope}

    \begin{scope}[shift={(20,0)}]
      \fill[blue!20,opacity=.4] ($(90:\VR) + (150:\Vr)$) -- ($(90:\VR) + (150:-\Vr)$) --
      ($(210:\VR) + (150:-\Vr)$) -- ($(210:\VR)+(150:\Vr)$);
      \fill[blue!20,opacity=.4] ($(90:\VR) + (30:\Vr)$) -- ($(90:\VR) + (30:-\Vr)$) --
      ($(330:\VR) + (30:-\Vr)$) -- ($(330:\VR)+(30:\Vr)$);
      \fill[blue!20,opacity=.4] ($(330:\VR) + (90:\Vr)$) -- ($(330:\VR) + (90:-\Vr)$) --
      ($(210:\VR) + (90:-\Vr)$) -- ($(210:\VR)+(90:\Vr)$);
      \draw[circle,fill=white] (90:\VR) circle (\Vr) {};
      \draw[circle,fill=white] (210:\VR) circle (\Vr) {};
      \draw[circle,fill=white] (330:\VR) circle (\Vr) {};

      \node[b] (x) at (80:2.5){};
      \node[b] (x') at (100:2.5){};
      \node[b] (y) at (200:2.5){};
      \node[b] (y') at (220:2.5){};
      \node[b] (z) at (320:2.5){};
      \node[b] (z') at (340:2.5){};

      \draw[e] (x)--(y)(x')(y')(x);
      \draw[e] (x)--(z)--(x')(z')--(x);
      \draw[e] (y)--(z)(y')--(z')--(y);
    \end{scope}
  \end{tikzpicture}
  \caption{Pseuodrandomness conditions for the relative Roth theorem.} \label{fig:rel-roth}
\end{figure}
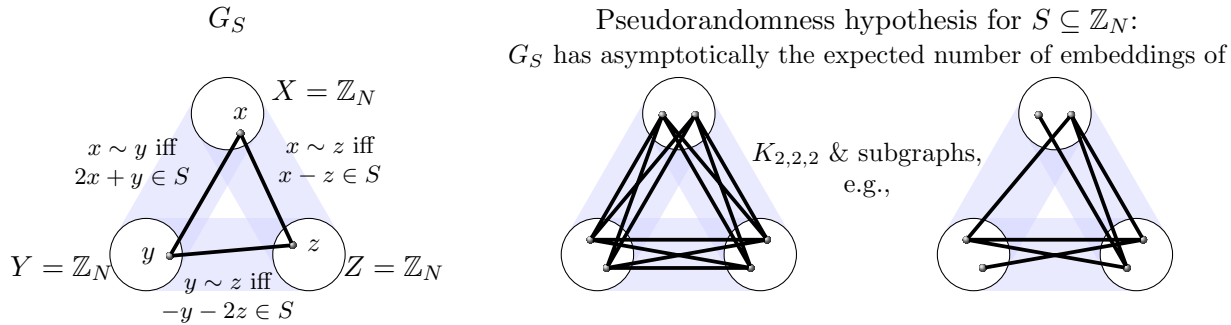

To state the pseudorandomness conditions, let $p = \abs{S}/N$ (which may decrease
as a function of $N$) and consider the graph $G_S$. This is similar to $G_A$,
except that $(x,y) \in X \x Y$ is now made an edge if and only if $2x + y
\in S$, etc. The pseudorandomness hypothesis now asks that the number of
embeddings of $K_{2,2,2}$ in $G_S$ (i.e., the number of tuples $(x_1,
x_2, y_1, y_2, z_1, z_2) \in X \x X \x Y \x Y \x Z \x Z$ where
$x_iy_j, x_iz_j, y_iz_j$ are all edges for all $i,j \in \{1,2\}$) be equal to
$(1 + o(1))p^{12} N^6$, where $o(1)$ indicates a quantity that tends to zero as $N$ tends to infinity.
This is asymptotically the same as the expected number of embeddings of $K_{2,2,2}$
in a random tripartite graph of density $p$ or in the graph $G_S$ formed from a random set $S$ of density $p$.
Assuming that $p$ does not decrease too rapidly with $N$, it is possible to show that with high probability the
true $K_{2,2,2}$-count in these random graphs is asymptotic to this expectation. It is therefore
appropriate to think of our condition as a type of pseudorandomness.

For technical reasons, it is necessary to assume that this property of having the ``correct'' count
holds not only for $K_{2,2,2}$ but also for every subgraph
$H$ of $K_{2,2,2}$. That is, we ask that the number of embeddings of $H$ into $G_S$
(with vertices of $H$ mapped into their assigned parts) be equal to $(1+ o(1))
p^{e(H)} N^{v(H)}$. The full description is now summarized in Figure~\ref{fig:rel-roth}, although we will
restate it in more formal terms later on.

\begin{figure}
   \begin{tikzpicture}[
     b/.style={ball color = gray,circle,inner sep = 0,minimum
       size=1mm},
     e/.style={line width=.5mm,line cap = rect}
     ]
     \begin{scope}
       \node[b] (a) at (0,0.25) {};
       \node[b] (b) at (1,0.25) {};
       \draw[e] (a)--(b);
       \node[anchor=center] at (2.5,.4) {$\xrightarrow{\text{2-blow-up}}$};
     \end{scope}

     \begin{scope}[xshift = 4cm]
       \node[b] (a) at (0,0) {};
       \node[b] (a') at (0,.5) {};
       \node[b] (b) at (1,0) {};
       \node[b] (b') at (1,.5) {};
       \draw[e] (a)--(b)--(a')--(b')--(a);
     \end{scope}

     \begin{scope}[xshift=8cm]
       \node[b] (x) at (90:.7){};
       \node[b] (y) at (210:.7){};
       \node[b] (z) at (330:.7){};
       \draw[e] (x)--(y)--(z)--(x);
             \node[anchor=center] at (2,.4) {$\xrightarrow{\text{2-blow-up}}$};
     \end{scope}

     \begin{scope}[xshift = 12cm]
     \node[b] (x) at (75:.7){};
     \node[b] (x') at (105:.7){};
     \node[b] (y) at (195:.7){};
     \node[b] (y') at (225:.7){};
     \node[b] (z) at (315:.7){};
     \node[b] (z') at (345:.7){};

     \draw[e] (x)--(y)--(x')--(y')--(x);
     \draw[e] (x)--(z)--(x')--(z')--(x);
     \draw[e] (y)--(z)--(y')--(z')--(y);
     \end{scope}

   \end{tikzpicture}
   \caption{The 2-blow-up of a graph is constructed by duplicating
     each vertex.} \label{fig:2-blow-up}
\end{figure}
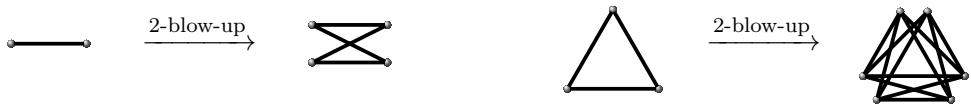

To see why this is a natural pseudorandomness hypothesis, we recall a famous result of Chung, Graham,
and Wilson \cite{CGW89}. This result says that several seemingly different notions of
pseudorandomness for dense graphs (i.e., graphs with constant edge
density) are equivalent. These notions are based, for example, on measuring eigenvalues, edge discrepancy, subgraph counts, or codegree
distributions. One rather striking fact is that having
the expected $4$-cycle count turns out to be equivalent to all of the other
definitions.

For sparse graphs, these equivalences do not hold. While having the correct count for $4$-cycles, which may be seen as the $2$-blow-up of an edge (see Figure~\ref{fig:2-blow-up}), still gives some control over the distribution of edges in the graph, this property is no longer strong enough to control the distribution of other small graphs such as triangles. This is where the pseudorandomness condition described above becomes useful, because knowing that we have approximately the correct count for $K_{2,2,2}$, the $2$-blow-up of a triangle, does allow one to control the distribution of triangles.

We now sketch the idea behind the proof of the relative Roth theorem. We begin by noting that Roth's theorem can be rephrased as follows.

\begin{theorem}[Roth] \label{thm:roth-delta}
  For every $\delta > 0$, every $A \subseteq
\ZZ_N$ with $\abs{A} \geq \delta N$ contains a $3$-AP, provided $N$ is
sufficiently large.
\end{theorem}

By a simple averaging argument (attributed to Varnavides \cite{Var59}), this version of
Roth's theorem is equivalent to the claim that $A$ contains not just one, but many $3$-APs.

\begin{theorem}
 [Roth's theorem, counting version] \label{thm:roth-counting}
For every $\delta > 0$, there exists $c = c(\delta) > 0$
  such that every $A \subseteq
\ZZ_N$ with $\abs{A} \geq \delta N$ contains at least $cN^2$ $3$-APs, provided $N$ is
sufficiently large.
\end{theorem}


To prove the relative Roth theorem from
Roth's theorem, assume that $A \subseteq S \subseteq \ZZ_N$ is such that
$\abs{A} \geq \delta \abs{S}$. The first step is to show that there is
a \emph{dense model} $\tA$ for $A$. This is a dense subset of
$\ZZ_N$ such that $\sabs{\tA}/N \approx |A|/|S| \geq \delta$
and $\tA$ approximates $A$ in the sense of a certain cut norm.
The second step is to use this cut norm condition to prove a counting
lemma, which says that $\tA$ and $A$ contain roughly the
same number of $3$-APs (after an appropriate normalization), i.e.,
\[
(N/\abs{S})^3 \abs{\{\text{$3$-APs in $A$}\}}
\approx
\sabs{\{\text{$3$-APs in $\tA$}\}}.
\]
Since the counting version of Roth's theorem implies that $\sabs{\{\text{$3$-APs in $\tA$}\}} \geq cN^2$, the relative Roth theorem is proved.

This discussion is fairly accurate, except for one white lie, which is that it is more correct to think of $\tA$ as a weighted function from $\ZZ_N$  to $[0,1]$ than as a subset of $\ZZ_N$. It will therefore be more convenient to work with the following weighted version of Roth's theorem. At this point, it is worth fixing some notation. We will write $\EE_{x_1 \in X_1, \dots, x_k \in X_k}$ as a shorthand for $|X_1|^{-1} \dots |X_k|^{-1} \sum_{x_1 \in X_1} \dots \sum_{x_k \in X_k}$. If the variables $x_1, \dots, x_k$ or the sets $X_1, \dots, X_k$ are understood, we will sometimes choose to omit them. We will also write $o(1)$ to indicate a function that tends to zero as $N$ tends to infinity, indicating further dependencies by subscripts when they are not understood.

\begin{theorem}
  [Roth's theorem, weighted version] \label{thm:roth-weighted}
 For every $\delta > 0$, there exists $c = c(\delta) > 0$
  such that every $f \colon \ZZ_N \to [0,1]$ with $\EE f \geq \delta$
  satisfies
  \begin{equation}\label{eq:roth-weighted}
    \EE_{x,d \in \ZZ_N}[f(x)f(x+d)f(x+2d)] \geq c - o_{\delta}(1).
  \end{equation}
\end{theorem}

Note that when $f$ is $\{0,1\}$-valued, i.e.,
$f = 1_A$ is the indicator function of some set $A$, this reduces to the counting
version of Roth's theorem. Up to a change of parameters, the counting version
also implies the weighted version. Indeed,
to deduce the weighted version from the counting version, let $A = \{x \in
  \ZZ_N \mid f(x) \geq \delta/2\}$. If $\EE f \geq \delta$ and $0 \leq
  f \leq 1$, then $\abs{A} \geq \delta N / 2$, so
\[
  \EE_{x,d \in \ZZ_N}[f(x)f(x+d)f(x+2d)] \geq (\delta/2)^3
  \EE_{x,d \in \ZZ_N}[1_A(x)1_A(x+d)1_A(x+2d)].
\]
By the counting version of Roth's theorem, this is bounded below by a positive constant when $N$ is sufficiently large.

When working in the functional setting, we also replace the set $S$ by a function $\nu \colon \ZZ_N \to
[0,\infty)$. This function $\nu$, which we call a
\emph{majorizing measure}, will be normalized to satisfy\footnote{We think of $\nu$ as
  a sequence of functions $\nu^{(N)}$, though we usually
  suppress the implicit dependence of $\nu$ on $N$.}
\[
\EE \nu = 1 + o(1).
\]
The subset $A \subseteq S$ will be replaced by some function
$f \colon \ZZ_N \to [0,\infty)$ majorized by $\nu$, that is, such that $0 \leq f(x)
\leq \nu(x)$ for all $x \in \ZZ_N$ (we write this as $0 \leq f \leq
\nu$). The hypothesis $\abs{A} \geq \delta |S|$ will be replaced
by $\EE f \geq \delta$. Note that $\nu$ and $f$ can be unbounded,
which is a major source of difficulty. The main motivating example to
bear in mind is that when $A \subseteq S \subseteq \ZZ_N$, we take
$\nu(x) = \frac{N}{\abs{S}} 1_S(x)$ and $f(x) = \nu(x)1_A(x)$, noting that
if $|A| \geq \delta |S|$ then $\EE f \geq \delta$.
We refer the reader to Table~\ref{tab:setting} for a summary of this correspondence.

\begin{table}
  \begin{tabular}{p{.1\textwidth}p{.15\textwidth}p{.22\textwidth}}
    \toprule
    & Sets &  Functions \\ \midrule
    Dense setting &
    $A \subseteq \ZZ_N$

    $\abs{A} \geq \delta N$
    &
    $f \colon \ZZ_N \to [0,1]$

    $\EE f \geq \delta$
    \\
    \midrule

    Sparse setting &

    $A \subseteq S \subseteq \ZZ_N$

    $\abs{A} \geq \delta \abs{S}$
    &
    $f \leq \nu \colon \ZZ_N \to [0,\infty)$

    $\EE f \geq \delta$, \ $\EE \nu = 1 + o(1)$ \\
    \bottomrule \\
  \end{tabular}
  \caption{Comparing the set version with the weighted version.} \label{tab:setting}
\end{table}

We can now state the pseudorandomness condition in a more formal way. We modify
the graph $G_S$ to a weighted graph $G_\nu$, which, for brevity, we usually
denote by $\nu$. This is a weighted
tripartite graph with vertex sets $X = Y = Z = \ZZ_N$ and edge
weights given by:
\begin{itemize}
\item $\nu_{XY}(x,y) = \nu(2x + y)$ for all $(x,y) \in X \x Y$;
\item $\nu_{XZ}(x,z) = \nu(x - z)$ for all $(x,z) \in X \x Z$;
\item $\nu_{YZ}(y,z) = \nu(-y-2z)$ for all $(y,z) \in Y \x Z$.
\end{itemize}
We will omit the subscripts if there is no risk of confusion. The
pseudorandomness condition then says that the weighted graph $\nu$ has
asymptotically the expected $H$-density for any subgraph $H$ of $K_{2,2,2}$. For example, triangle density in $\nu$
is given by the expression $\EE[\nu(x,y)\nu(x,z)\nu(y,z)]$, where
$x$, $y$, $z$ vary independently and uniformly over $X$, $Y$, $Z$, respectively.
The pseudorandomness assumption requires, amongst other things, that this triangle density
be $1 + o(1)$, the normalization having accounted for the other factors. The
full hypothesis, involving $K_{2,2,2}$ and its subgraphs, is stated below.

\begin{definition}[$3$-linear forms condition] \label{def:3-lfc}
  A weighted tripartite graph $\nu$ with vertex sets
  $X$, $Y$, and $Z$ satisfies the \emph{$3$-linear forms condition} if
  \begin{multline} \label{eq:3-lfc-graph}
    \EE_{x,x' \in X,\ y,y' \in Y,\ z,z' \in Z}
    [\nu(y,z) \nu(y',z) \nu(y,z') \nu(y',z')
    \nu(x,z) \nu(x',z) \nu(x,z') \nu(x',z') \\
    \cdot \nu(x,y) \nu(x',y) \nu(x,y') \nu(x',y')] = 1 + o(1)
  \end{multline}
  and also \eqref{eq:3-lfc-graph} holds when one or more of the twelve $\nu$
  factors in the expectation are erased.

  Similarly, a function $\nu \colon \ZZ_N \to [0,\infty)$ satisfies
  the \emph{$3$-linear forms condition}\footnote{We will
    assume that $N$ is odd, which simplifies the proof of
    Theorem~\ref{thm:rel-roth} without too much loss in generality. Theorem~\ref{thm:rel-roth} holds more generally without this additional
    assumption on $N$.} if
  \begin{multline} \label{eq:3-lfc}
    \EE_{x,x',y,y',z,z' \in \ZZ_N} [\nu(-y-2z) \nu(-y'-2z) \nu(-y-2z')
    \nu(-y'-2z')
    \nu(x-z) \nu(x'-z) \\ \cdot  \nu(x-z') \nu(x'-z')
    \nu(2x+y) \nu(2x'+y) \nu(2x+y') \nu(2x'+y')] = 1 + o(1)
  \end{multline}
  and also \eqref{eq:3-lfc} holds when one or more of the twelve $\nu$
  factors in the expectation are erased.
\end{definition}

\begin{remark}
The $3$-linear forms condition \eqref{eq:3-lfc} for a function $\nu: \ZZ_N \to [0, \infty)$ is precisely the same as \eqref{eq:3-lfc-graph} for the weighted graph $G_\nu$.
\end{remark}

We can now state the relative Roth theorem formally.

\begin{theorem} [Relative Roth]
  \label{thm:rel-roth}
  Suppose $\nu \colon \ZZ_N \to [0, \infty)$ satisfies the $3$-linear forms condition. For every $\delta > 0$, there exists $c = c(\delta) > 0$ such that every $f \colon \ZZ_N \to [0,\infty)$ with $0 \leq f \leq \nu$ and $\EE f \geq \delta$ satisfies
  \begin{equation*} \label{eq:rel-roth}
  \EE_{x,d\in\ZZ_N}[f(x)f(x+d)f(x + 2d)]
  \geq c - o_{\delta}(1).
\end{equation*}
Moreover, $c(\delta)$ may be taken to be the same constant which appears in
\eqref{eq:roth-weighted}.
\end{theorem}

\begin{remark}The rate at which the $o(1)$ term in \eqref{eq:rel-roth} goes
to zero depends not only on $\delta$ but also on the rate of
convergence in the $3$-linear forms condition.
\end{remark}

\section{Relative Szemer\'edi theorem} \label{sec:rel-sz}

As in the case of Roth's theorem, we first state an equivalent version
of Szemer\'edi's theorem allowing weights.

\begin{theorem}
  [Szemer\'edi's theorem, weighted version]
  \label{thm:sz-weighted}
  For every $k \geq 3$ and $\delta > 0$, there exists $c = c(k, \delta) > 0$ such that every $f \colon \ZZ_N \to [0,1]$
  with $\EE f \geq \delta$ satisfies
  \begin{equation} \label{eq:sz-weighted}
    \EE_{x,d \in \ZZ_N}[ f(x)f(x+d)f(x + 2d) \cdots f(x+(k-1)d)]
    \geq c - o_{k,\delta}(1).
  \end{equation}
\end{theorem}

The setup for the relative Szemer\'edi theorem is a natural extension of
the previous section. Just as our pseudorandomness condition for $3$-APs
was related to the graph-theoretic approach to Roth's theorem, the pseudorandomness
condition in the general case is informed by the hypergraph removal approach to
Szemer\'edi's theorem~\cite{Gow07, NRS06, RS04, RS06, Tao06jcta}.

Instead of constructing a weighted graph as we
did for $3$-APs, we now construct a weighted $(k-1)$-uniform hypergraph
corresponding to $k$-APs. For example, for $4$-APs, the $3$-uniform
hypergraph corresponding to the majorizing measure $\nu \colon \ZZ_N
\to [0,\infty)$ is $4$-partite, with vertex sets $W, X, Y, Z$, each with
$N$ vertices labeled by elements of $\ZZ_N$. The weighted edges are
given by:
\begin{itemize}
\item $\nu_{WXY}(w,x,y) = \nu(3w+2x+y)$ on $W \x X \x Y$;
\item $\nu_{WXZ}(w,x,z) = \nu(2w+x-z)$ on $W \x X \x Z$;
\item $\nu_{WYZ}(w,y,z) = \nu(w-y-2z)$ on $W \x Y \x Z$;
\item $\nu_{XYZ}(x,y,z) = \nu(-x-2y-3z)$ on $X \x Y \x Z$.
\end{itemize}
The linear forms $3w+2x+y, 2w+x-z, w-y-2z, -x-2y-3z$ are chosen
because they form a 4-AP with common difference $-w-x-y-z$ and each
linear form depends on exactly three of the four variables. The
pseudorandomness condition then says that the weighted hypergraph $\nu$
contains asymptotically the expected count of $H$ whenever $H$ is a
subgraph of the $2$-blow-up of the simplex $K_4^{(3)}$. Here $K_4^{(3)}$
is the complete $3$-uniform hypergraph on $4$ vertices, that is, with vertices
$\{w,x,y,z\}$ and edges $\{wxy,wxz,wyz,xyz\}$, while the $2$-blow-up of $K_4^{(3)}$ is the
$3$-uniform hypergraph constructed by duplicating each vertex in $K_4^{(3)}$ and joining
all those triples which correspond to edges in $K_4^{(3)}$. Explicitly, this $2$-blow-up has
vertex set $\{w_1, w_2, x_1, x_2, y_1, y_2, z_1, z_2\}$ and edges
$w_i x_j y_k, w_i x_j z_k, w_i y_j z_k, x_i y_j z_k$ for all $i, j, k \in \{1, 2\}$.

For general $k$, we are concerned with $K_k^{(k-1)}$, the complete $(k-1)$-uniform hypergraph on $k$ vertices, while the
pseudorandomness condition again asks that a certain weighted $k$-partite $(k-1)$-uniform hypergraph contains
asymptotically the expected count for every subgraph of the $2$-blow-up of $K_k^{(k-1)}$. This $2$-blow-up is
constructed analogously to the $2$-blow-up of $K_4^{(3)}$ above and has $2k$ vertices and $k 2^{k-1}$ edges.

For $k$-APs, the corresponding linear forms are given by the expressions $\sum_{i=1}^k
(j-i) x_i$, for each $j = k, k-1, \dots, 1$. The condition~\eqref{eq:Z-lfc} below is now the natural
extension of the $3$-linear forms condition~\eqref{eq:3-lfc}. When viewed as a hypergraph
condition, it asks that the count for any subgraph of the $2$-blow-up of $K_k^{(k-1)}$ be
close to the expected count.

\begin{definition}[Linear forms condition]
  \label{def:Z-lfc}
  A function $\nu : \ZZ_N \to [0,\infty)$ satisfies the
  \emph{$k$-linear forms condition}\footnote{As in the footnote to
    Definition~\ref{def:3-lfc}, in our proof of
    Theorem~\ref{thm:rel-sz} we will make the simplifying assumption
    that $N$ is coprime to
    $(k-1)!$. In the proof of the Green-Tao
    theorem, one can always make this assumption.} if
  \begin{equation}
    \label{eq:Z-lfc}
    \EE_{x_1^{(0)}, x_1^{(1)}, \dots, x_k^{(0)},x_k^{(1)} \in \ZZ_N} \Bigl[
    \prod_{j=1}^k \prod_{\omega \in \{0,1\}^{[k]\setminus\{j\}}} \nu
    \Bigl( \sum_{i=1}^k (j-i) x_i^{(\omega_i)} \Bigr)^{n_{j,\omega}}
    \Bigr]  = 1 + o(1)
  \end{equation}
  for any choice of exponents $n_{j,\omega} \in \{0,1\}$.
\end{definition}

Now we are ready to state the main result in the proof of the
Green-Tao theorem.

\begin{theorem}
  [Relative Szemer\'edi]
  \label{thm:rel-sz}
  Suppose $k \geq 3$ and $\nu \colon \ZZ_N \to [0, \infty)$ satisfies the $k$-linear forms condition. For every $\delta > 0$, there exists $c = c(k, \delta) > 0$ such that every $f \colon \ZZ_N \to [0,\infty)$ with $0 \leq f \leq \nu$ and $\EE f \geq \delta$ satisfies
  \begin{equation} \label{eq:rel-sz}
  \EE_{x,d \in \ZZ_N}[ f(x)f(x+d)f(x + 2d) \cdots f(x+(k-1)d)]
  \geq c - o_{k, \delta}(1).
\end{equation}
Moreover, $c(k,\delta)$ may be taken to be the same constant which appears in
\eqref{eq:sz-weighted}.
\end{theorem}

\begin{remark}The rate at which the $o(1)$ term in \eqref{eq:rel-sz} goes
to zero depends not only on $k$ and $\delta$ but also on the rate of
convergence in the $k$-linear forms condition for $\nu$.
\end{remark}

Now we outline the proof of the relative Szemer\'edi theorem. This is
simply a rephrasing of the outline given after
Theorem~\ref{thm:roth-counting} for the unweighted version of the relative
Roth theorem. We start with $0 \leq f \leq \nu$ and $\EE f \geq
\delta$. In Section \ref{sec:dense-model}, we prove a \emph{dense model
  theorem} which shows that there exists another function $\tf \colon
\ZZ_N \to [0,1]$ which approximates $f$ with respect to a certain cut
norm.\footnote{In the original Green-Tao approach, they required
  $\tf$ and $f$ to be close in a stronger sense related to the Gowers
  uniformity norm. The cut norm approach we present
  here requires less stringent pseudorandomness hypotheses for applying the
  dense model theorem but a stronger counting lemma.} Note that $\tf$ is bounded (hence ``dense'' model) and $\EE \tf \geq \delta - o(1)$. In
Section \ref{sec:count}, we establish a \emph{counting lemma} which says that
the weighted $k$-AP counts in $f$ and $\tf$ are similar, that is,
\begin{equation*} \label{eq:outline-Z-counting}
  \EE_{x,d}[f(x)f(x+d)\cdots f(x+(k-1)d)]
  =
  \EE_{x,d}[\tf(x)\tf(x+d)\cdots \tf(x+(k-1)d)] - o(1).
\end{equation*}
The right-hand side is at least $c(k,\delta) - o_{k, \delta}(1)$ by Szemer\'edi's theorem
(Theorem~\ref{thm:sz-weighted}). Thus the relative Szemer\'edi
theorem follows. We now begin the proof proper.

\section{Dense model theorem} \label{sec:dense-model}

Given $g \colon X \x Y \to \RR$, viewed as an edge-weighted bipartite
graph with vertex set $X \cup Y$, the \emph{cut norm} of
$g$, introduced by Frieze and Kannan~\cite{FK99} (also see
\cite[Chapter 8]{L12}), is defined as
\begin{equation} \label{eq:graph-cut}
\norm{g}_\square := \sup_{A \subseteq X, B \subseteq Y} \abs{\EE_{x
    \in X, y \in Y} [g(x,y) 1_A(x)1_B(y)]}.
\end{equation}
For a weighted 3-uniform hypergraph $g \colon X \x Y \x Z \to \RR$, we define
\[
\norm{g}_\square := \sup_{A \subseteq Y \x Z,\ B \subseteq X \x Z,\ C
  \subseteq X \x Y} \abs{\EE_{x \in X, y \in Y, z \in Z} [g(x,y,z) 1_A(y,z)1_B(x,z)1_C(x,y)]}.
\]
(The more obvious alternative, where we range $A, B, C$ over subsets
of $X, Y, Z$, respectively, gives a weaker norm that is not sufficient
to guarantee a counting lemma.) More generally, given a weighted $r$-uniform
hypergraph $g \colon X_1 \x \cdots \x X_r \to \RR$, define
\begin{equation} \label{eq:hyp-cut-norm}
\norm{g}_\square := \sup \abs{\EE_{x_1 \in X_1, \dots, x_r \in X_r} [
  g(x_1, \dots, x_r) 1_{A_1}(x_{-1}) 1_{A_2}(x_{-2}) \cdots 1_{A_r} (x_{-r})]},
\end{equation}
where the supremum is taken over all choices of subsets $A_i \subseteq
X_{-i} := \prod_{j\in[r] \setminus \{i\}} X_j$, $i  \in [r]$, and we write
\[
x_{-i} := (x_1, x_2, \dots, x_{i-1}, x_{i+1}, \dots, x_r) \in X_{-i}
\]
for each $i$. We extend this definition of cut norm to $\ZZ_N$: for
any function $f \colon \ZZ_N \to \RR$, define
\begin{equation} \label{eq:Z-cut-norm} \norm{f}_{\square,r}:= \sup
  \abs{\EE_{x_1, \dots, x_r \in \ZZ_N} [f(x_1 + \cdots + x_r)
    1_{A_1}(x_{-1}) 1_{A_2}(x_{-2}) \cdots 1_{A_r} (x_{-r})]},
\end{equation}
where the supremum is taken over all $A_1, \dots, A_r \subseteq \ZZ_N^{r-1}$.
It is easy to see that this is a norm. Equivalently, it is the
hypergraph cut norm applied to the weighted $r$-uniform hypergraph $g \colon
\ZZ_N^r \to \RR$ with $g(x_1, \dots, x_r) = f(x_1 + \cdots +
x_r)$. For example,
\[
\norm{f}_{\square, 2} := \sup_{A, B \subseteq \ZZ_N}
\abs{\EE_{x,y\in\ZZ_N} [ f(x + y) 1_A(x)1_B(y)]}.
\]

The main result of this section is the following dense model theorem
(in this particular form due to the third author \cite{Zhao14}). It gives a condition under which it is possible to
approximate an unbounded (or sparse) function $f$ by a bounded
(or dense) function $\tf$.

\begin{theorem}[Dense model] \label{thm:dense-model}
  For every $\e > 0$, there exists an $\e' >
  0$ such that the following holds. Suppose $\nu \colon \ZZ_N \to [0,
  \infty)$ satisfies $\snorm{\nu - 1}_{\square,r} \leq \e'$. Then, for
  every $f \colon \ZZ_N \to [0,\infty)$ with $f \leq \nu$, there exists a function $\tf \colon \ZZ_N \to
  [0,1]$ such that $\snorm{f - \tf}_{\square,r} \leq \e$.
\end{theorem}

\begin{remark} One may take $\e' = \exp(-\e^{-C})$ where $C$ is some
    absolute constant (independent of $r$ and, more importantly,
    $N$).
  \end{remark}

 A more involved dense model theorem (using a norm based on the
  Gowers uniformity norm rather than the cut norm) was used by Green and Tao in
  \cite{GT08}. Its proof was subsequently simplified
  by Gowers \cite{Gow10} and, independently, Reingold, Trevisan, Tulsiani, and Vadhan \cite{RTTVnote}. Here we follow
  Gowers' approach, but specialized to $\norm{\cdot}_{\square,r}$, which simplifies the exposition.

It will be useful to rewrite
$\EE_{x,y} [f(x+y)1_A(x) 1_B(y)]$
in the form $\ang{f,\varphi} = \EE_x [f(x)\varphi(x)]$ for some
$\varphi \colon \ZZ_N \to \RR$. We have, by a change of variable,
\[
\EE_{x,y} [f(x+y)1_A(x) 1_B(y)]
= \EE_{x,z} [f(z)1_A(x)1_B(z-x)]
= \ang{f, 1_A * 1_B},
\]
where the convolution is defined by $h_1 * h_2(z) := \EE_x
[h_1(x)h_2(z-x)]$. Let $\Phi_2$ denote the set of all functions that
can be written as a convex combination of convolutions $1_A * 1_B$
with $A, B \subseteq \ZZ_N$. We then have, by convexity,
\[
\norm{f}_{\square,2}
= \sup_{A,B \subseteq \ZZ_N} \abs{\ang{f,1_A * 1_B}} = \sup_{\varphi \in \Phi_2} \sabs{\sang{f, \varphi}}.
\]
More generally, given $r$ functions $h_1, \dots, h_r \colon
\ZZ_N^{r-1} \to \RR$, define their generalized convolution $(h_1, \dots, h_r)^*
 \colon \ZZ_N \to \RR$ by
\[
(h_1, \dots, h_r)^*(x) = \EE_{\substack{y_1, \dots, y_r \in \ZZ_N \\
    y_1 + \cdots + y_r = x}} [
h_1(y_2,\cdots,y_r)h_2(y_1,y_3,\dots,y_r)\cdots h_r(y_1,\cdots,y_{r-1})].
\]
For example, when $r = 2$, we recover the usual convolution $(h_1,
h_2)^* = h_1 * h_2$. We similarly have
\begin{equation*} \label{eq:cut-inner}
\norm{f}_{\square,r}
= \sup_{A_1, \dots, A_r \subseteq \ZZ_N^{r-1}} \abs{\ang{f,(1_{A_1},\dots,1_{A_r})^*}}
= \sup_{\varphi \in \Phi_r} \sabs{\sang{f, \varphi}},
\end{equation*}
where $\Phi_r$ is the set of all functions $\varphi: \ZZ_N \to \RR$ that can be
written as a convex combination of generalized convolutions $(1_{A_1},
1_{A_2}, \dots, 1_{A_r})^*$ with $A_1, \dots, A_r
\subseteq \ZZ_N^{r-1}$. The next lemma establishes a key property of $\Phi_r$.

\begin{lemma} \label{lem:mul-closed}
  The set $\Phi_r$ is closed under multiplication,
  i.e., if $\varphi, \varphi' \in \Phi_r$, then $\varphi\varphi' \in
  \Phi_r$.
\end{lemma}

\begin{proof}
  It suffices to show that if $\varphi = (1_{A_1}, \cdots, 1_{A_r})^*$
  and $\varphi' = (1_{B_1}, \dots, 1_{B_r})^*$, where $A_1, \dots, A_r$,
  $B_1, \dots, B_r \subseteq \ZZ_N^{r-1}$, then $\varphi\varphi' \in
  \Phi_r$. For any $y = (y_1, \dots, y_r) \in \ZZ_N^r$, we write $\Sigma y = y_1 + \dots + y_r$ and
  $y_{-i} = (y_1, \dots, y_{i-1},y_{i+1}, \dots, y_r) \in
  \ZZ_N^{r-1}$. Then, for any $x \in \ZZ_N$, we have
  \begin{align*}
    \varphi(x)\varphi'(x)
    &= \EE_{\substack{y,y' \in \ZZ_N^r \\ \Sigma y = \Sigma y' = x}}
    [1_{A_1}(y_{-1}) 1_{B_1}(y'_{-1}) \cdots 1_{A_r}(y_{-r})
    1_{B_r}(y'_{-r})]
    \\
    &= \EE_{\substack{y,z \in \ZZ_N^r \\ \Sigma y = x, \Sigma z = 0}}
    [1_{A_1}(y_{-1}) 1_{B_1}(y_{-1} + z_{-1}) \cdots 1_{A_r}(y_{-r})
    1_{B_r}(y_{-r} + z_{-r})]
    \\
    &= \EE_{\substack{y,z \in \ZZ_N^r \\ \Sigma y = x, \Sigma z = 0}}
    [1_{A_1 \cap (B_1 - z_{-1})}(y_{-1}) \cdots 1_{A_r \cap (B_r - z_{-r})}(y_{-r})]
    \\
    &= \EE_{\substack{z \in \ZZ_N^r \\ \Sigma z = 0}} [(1_{A_1 \cap
      (B_1 - z_{-1})}, \dots, 1_{A_r \cap (B_r - z_{-r})})^*(x)].
  \end{align*}
This expresses $\varphi\varphi'$ as a convex combination of
generalized convolutions. Thus $\varphi \varphi' \in \Phi_r$.
\end{proof}

For the rest of this section, we fix the value of $r$ and simply write
$\norm{\cdot }$ for $\norm{\cdot }_{\square,r}$ and $\Phi$ for
$\Phi_r$. We have $\norm{f} = \sup_{\varphi \in \Phi}
\sabs{\sang{f,\varphi}}$. An important role in the proof is played by the
dual norm, which is defined by $\norm{\psi}^* = \sup_{\norm{f} \leq 1} \ang{f, \psi}$.
It follows easily from the definition that $\sabs{\sang{f,\psi}} \leq \norm{f} \norm{\psi}^*$.

It is also easy to show that the unit ball for this dual norm is the convex hull
of the union of $\Phi$ and $-\Phi$. To see that the convex hull is contained in the unit ball, we note that each element of $\Phi \cup (-\Phi)$ is in the unit ball and apply the triangle inequality to deduce that the same holds for convex combinations. For the reverse implication, suppose that $\psi$ is in the unit ball of $\norm{\cdot}^*$ but not in the convex hull of $\Phi \cup (-\Phi)$. Then, by the separating hyperplane theorem, there exists $f$ such that $\sabs{\sang{f, \varphi}} \leq 1$ for all $\varphi \in \Phi \cup (-\Phi)$ and $\sang{f, \psi} > 1$. But the first inequality implies that $\norm{f} \leq 1$ and so, by the second inequality, $\norm{\psi}^* > 1$, contradicting our assumption. By Lemma~\ref{lem:mul-closed}, this now implies that the unit ball for the dual norm is closed under multiplication. Thus, for every $\varphi, \psi \colon \ZZ_N \to
\RR$, we have $\norm{(\varphi/\norm{\varphi}^*)(\psi/\norm{\psi}^*)}^*
\leq 1$, i.e.,
\begin{equation}
  \label{eq:dual-norm-mul}
  \norm{\varphi\psi}^* \leq \norm{\varphi}^*\norm{\psi}^*.
\end{equation}

Finally, we note that $\norm{\cdot } \leq
\norm{\cdot }_1$ and $\norm{\cdot }_\infty \leq \norm{\cdot }^*$.
The first inequality follows since
\[\norm{f} =  \sup_{\varphi \in \Phi}
\sabs{\sang{f,\varphi}} = \sup_{\varphi \in \Phi}
\sabs{\EE_x f(x) \varphi(x)} \leq \EE_x |f(x)| = \norm{f}_1.\]
The second inequality follows from duality or by letting $x'$ be a value for which $\psi$
achieves its maximum and taking $f(x) = N$ for $x = x'$ and $0$ otherwise.
It is then straightforward to verify that this function satisfies $\norm{f} \leq 1$ and
\[\norm{\psi}^* \geq \sabs{\sang{f, \psi}} = \sabs{\psi(x')} = \norm{\psi}_\infty.\]




\begin{proof}[Proof of Theorem~\ref{thm:dense-model}]
We may assume without loss of generality that $\epsilon \leq \frac{1}{10}$.
It suffices to show that there exists a function $\tf \colon \ZZ_N \to
[0,1+\e/2]$ with $\snorm{f - \tf} \leq \e/2$. Suppose, for
contradiction, that no such $\tf$ exists. Let
\[
K_1 := \{\tf \colon \ZZ_N \to [0,1+\e/2]\} \quad \text{and} \quad
K_2 := \{h \colon \ZZ_N \to \RR \mid \norm{h} \leq \e/2\}.
\]
We can view $K_1$ and $K_2$ as closed convex sets in
$\RR^N$. By assumption, $f \notin K_1 + K_2 := \{\tf + h : \tf
\in K_1, h \in K_2\}$. Therefore, since $K_1 + K_2$ is convex, the separating hyperplane
theorem implies that there exists some $\psi \colon \ZZ_N \to \RR$ such that
\begin{enumerate}
\item[(a)] $\ang{f, \psi} > 1$, and
\item[(b)] $\ang{g, \psi} \leq 1$ for all $g \in K_1 + K_2$.
\end{enumerate}
Note that since $0 \in K_1, K_2$, we have $K_1, K_2 \subset K_1 + K_2$. Therefore, in (b), we may take $g = (1+\e/2) 1_{\psi > 0} \in K_1$, obtaining $\ang{1, \psi_+} \leq
  (1 + \e/2)^{-1}$. Here $x_+ := \max\{0, x\}$ and $\psi_+(x) := \psi(x)_+$.
On the other hand, ranging $g$ over $K_2$, we obtain
$\norm{\psi}_\infty \leq \norm{\psi}^*
\leq 2/\e$, since if $\sang{g, \psi} \leq 1$ for all $g$ with $\norm{g} \leq \e/2$, then
$\sang{g, \psi} \leq 2/\e$ for all $g$ with $\norm{g} \leq 1$.

By the Weierstrass polynomial approximation theorem, there exists some
polynomial $P$ such that
$\abs{P(x) - x_+} \leq \e/8$ for all $x \in [-2/\e,2/\e]$. Let $P(x) =
p_dx^d + \cdots + p_1 x + p_0$ and $R = \abs{p_d} (2/\e)^d + \cdots +
\abs{p_1}(2/\e) + \abs{p_0}$ (it is possible to take $P$ so that $R =
\exp(\e^{-O(1)})$).

We write $P\psi$ to mean the function on $\ZZ_N$ defined by $P\psi(x)
= P(\psi(x))$. Using the triangle inequality,
\eqref{eq:dual-norm-mul}, and $\norm{\psi}^* \leq 2/\e$, we have
\[
\norm{P\psi}^*
\leq \sum_{i=0}^d \abs{p_i}\snorm{\psi^i}^*
\leq \sum_{i=0}^d \abs{p_i}(\snorm{\psi}^*)^i
\leq \sum_{i=0}^d \abs{p_i}(2/\e)^i
= R.
\]
Therefore, since we are assuming that $\norm{\nu-1} \leq \e'$,
\[
\abs{\ang{\nu-1,P\psi}} \leq \norm{\nu-1} \norm{P\psi}^* \leq \e' R.
\]
Since $\norm{\psi}_\infty \leq 2/\e$, we have $\norm{P\psi - \psi_+}_\infty
\leq \e/8$. Hence,
\[
\ang{\nu,P\psi} \leq
\ang{1,P\psi}  + \e'R
\leq \ang{1,\psi_+} + \e/8  + \e'R
\leq
(1+\e/2)^{-1} + \e/8 + \e'R.
\]
Also, we have $\norm{\nu}_1 = \ang{\nu, 1} \leq \norm{\nu - 1}
+ 1 \leq 1 + \e'$, where we used $\ang{\nu - 1, 1} \leq \norm{\nu - 1} \norm{1}^*$ and $\norm{1}^* = 1$. Thus,
\[
\ang{f,\psi} \leq \ang{f,\psi_+} \leq \ang{\nu,\psi_+}
\leq \ang{\nu,P\psi} + \norm{\nu}_1\norm{P\psi - \psi_+}_\infty
\leq (1+\e/2)^{-1} + \e/8 + \e'R + (1 + \e')\e/8.
\]
Since $\e \leq \frac{1}{10}$, the right-hand side is at most $1$ when $\e'$ is made sufficiently small (e.g.,
$\e' = \e/(8R)$), but this contradicts (a) from earlier. The dense
model theorem follows.
\end{proof}

\section{Counting lemma} \label{sec:count}

In this section, we prove the counting lemma. We will focus principally on the graph case,
Theorem~\ref{thm:tri-count-sparse} below,
since this case contains all the important ideas and is notationally simpler.
The hypergraph generalization is then discussed towards the end of the section.

For graphs, the counting lemma says that if two weighted graphs
are close in cut norm, then they have similar triangle densities.
To be more specific, we consider weighted tripartite
graphs on the vertex set $X \cup Y \cup Z$, where $X$, $Y$, and $Z$ are
finite sets. Such a weighted graph $g$ is given by three functions
$g_{XY} \colon X \x Y \to \RR$, $g_{XZ} \colon X \x Z \to \RR$, and
$g_{YZ} \colon Y \x Z \to \RR$, although we often drop the subscripts
if they are clear from context. We write $\norm{g}_\square =
\max\{\norm{g_{XY}}_\square, \norm{g_{XZ}}_\square,
\norm{g_{YZ}}_\square\}$.

We first consider the easier case of counting in dense (i.e., bounded
weight) graphs (see, for example, \cite{L12}).

\begin{proposition}[Triangle counting lemma, dense setting]
  \label{prop:tri-count-dense}
  Let $g$ and $\tg$ be weighted tripartite graphs on $X \cup Y \cup
  Z$ with weights in $[0,1]$. If $\norm{g - \tg}_\square \leq \e$, then
  \[
  \abs{\EE_{x \in X, y \in Y, z \in Z}[g(x,y)g(x,z)g(y,z) -
    \tg(x,y)\tg(x,z)\tg(y,z)]} \leq 3\e.
  \]
\end{proposition}

\begin{proof}
  Unless indicated otherwise, all expectations are taken over $x \in X$, $y \in Y$, $z \in
  Z$ uniformly and independently.
  From the definition \eqref{eq:graph-cut} of the cut norm, we have
  that   \begin{equation} \label{eq:graph-cut-relaxed}
  \abs{\EE_{x \in X, y \in Y}[(g(x,y)-\tg(x,y))a(x)b(y)]}\leq \e
  \end{equation}
  for every function $a \colon X \to [0,1]$ and $b \colon Y \to [0,1]$
  (since the expectation is bilinear in $a$ and $b$, the extrema
  occur when $a$ and $b$ are $\{0,1\}$-valued, so
  \eqref{eq:graph-cut-relaxed} is equivalent to
  \eqref{eq:graph-cut}). It follows that
  \[
  \abs{\EE[g(x,y)g(x,z)g(y,z) -
    \tg(x,y)g(x,z)g(y,z)]} \leq \e,
  \]
  since the expectation has the form \eqref{eq:graph-cut-relaxed} if
  we fix any value of $z$. Similarly, we have
  \[
  \abs{\EE[\tg(x,y)g(x,z)g(y,z) -
    \tg(x,y)\tg(x,z)g(y,z)]} \leq \e
  \]
  and
  \[
  \abs{\EE[\tg(x,y)\tg(x,z)g(y,z) -
    \tg(x,y)\tg(x,z)\tg(y,z)]} \leq \e.
  \]
  The result then follows from telescoping and the triangle inequality.
\end{proof}

This proof does not work in the sparse setting, when $g$ is unbounded,
since \eqref{eq:graph-cut-relaxed} requires $a$ and $b$ to be
bounded. The main result of this section, stated next for graphs
(the hypergraph version is stated towards the end of the section), gives a
counting lemma assuming $0 \leq g \leq \nu$ for some $\nu$ satisfying
the linear forms condition. This is one of the main results in our paper~\cite{CFZrelsz}.

\begin{theorem}[Relative triangle counting lemma]
  \label{thm:tri-count-sparse}
  Let $\nu, g, \tg$ be weighted tripartite graphs on $X \cup Y \cup
  Z$. Assume that $\nu$ satisfies the $3$-linear forms condition
  (Definition~\ref{def:3-lfc}), $0 \leq g \leq \nu$, and $0 \leq \tg
  \leq 1$. If $\norm{g - \tg}_\square
  = o(1)$, then
  \[
  \abs{\EE_{x \in X, y \in Y, z \in Z}[g(x,y)g(x,z)g(y,z) -
    \tg(x,y)\tg(x,z)\tg(y,z)]} = o(1).
  \]
\end{theorem}

\begin{figure}

  \newcommand{\VR}{2.5cm}
  \newcommand{\Vr}{2.5em}

  \begin{tikzpicture}[scale=.5,
    b/.style={ball color = gray,circle,inner sep = 0,minimum
      size=1mm},
    e/.style={line width=.5mm,line cap = rect}
    ]

    \begin{scope}[shift={(-2,0)}]
      \fill[blue!20,opacity=.3] ($(90:2.5) + (30:2.5em)$) -- ($(90:2.5) + (30:-2.5em)$) --
      ($(330:2.5) + (30:-2.5em)$) -- ($(330:2.5)+(30:2.5em)$);
      \fill[blue!20,opacity=.3] ($(330:2.5) + (90:2.5em)$) -- ($(330:2.5) + (90:-2.5em)$) --
      ($(210:2.5) + (90:-2.5em)$) -- ($(210:2.5)+(90:2.5em)$);
      \draw[circle,fill=white] (90:2.5) circle (2.5em) {};
      \draw[circle,fill=white] (210:2.5) circle (2.5em) {};
      \draw[circle,fill=white] (330:2.5) circle (2.5em) {};

      \node[b,label=above:{\small $x$}] (x) at (80:2){};
      \node[b,label=left:{\small $y$}] (y) at (220:2){};
      \node[b,label=below:{\small $z$}] (z) at (335:1.8){};
      \node[b,label=right:{\small $z'$}] (z') at (325:3.2){};

      \draw[e] (x)--(z)--(y)--(z')--(x);
             \node[font=\footnotesize,align=center] at (25:2)
       {$g$};
       \node[font=\footnotesize,align=center] at (-90:2)
       {$g$};
    \end{scope}

    \node at (4,-3.3) {$\EE[g(x,z)g(x,z')g(y,z)g(y,z')] = \EE[g'(x,y)g(x,z)g(y,z)]$};

        \begin{scope}[shift={(10,0)}]
      \fill[blue!20,opacity=.3] ($(90:2.5) + (30:2.5em)$) -- ($(90:2.5) + (30:-2.5em)$) --
      ($(330:2.5) + (30:-2.5em)$) -- ($(330:2.5)+(30:2.5em)$);
      \fill[blue!20,opacity=.3] ($(330:2.5) + (90:2.5em)$) -- ($(330:2.5) + (90:-2.5em)$) --
      ($(210:2.5) + (90:-2.5em)$) -- ($(210:2.5)+(90:2.5em)$);
      \fill[blue!70,opacity=.3] ($(90:2.5) + (150:2.5em)$) -- ($(90:2.5) + (150:-2.5em)$) --
      ($(210:2.5) + (150:-2.5em)$) -- ($(210:2.5)+(150:2.5em)$);
      \draw[circle,fill=white] (90:2.5) circle (2.5em) {};
      \draw[circle,fill=white] (210:2.5) circle (2.5em) {};
      \draw[circle,fill=white] (330:2.5) circle (2.5em) {};

      \node[b,label=above:{\small $x$}] (x) at (80:2){};
      \node[b,label=left:{\small $y$}] (y) at (220:2){};
      \node[b,label=below:{\small $z$}] (z) at (335:1.8){};

      \draw[e] (x)--(z)--(y);
             \node[font=\footnotesize,align=center] at (140:1.5)
       {$g'$};
       \node[font=\footnotesize,align=center] at (160:5.3)
       {$g'(x,y) = $ \\  $\EE_{z'\in Z}[g(x,z')g(y,z')]$};
       \node[font=\footnotesize,align=center] at (25:2)
       {$g$};
       \node[font=\footnotesize,align=center] at (-90:2)
       {$g$};
            \draw[e] (x)--(y);
    \end{scope}
  \end{tikzpicture}
  \vspace{-1em}
  \caption{The densification step in the proof of the relative
    triangle counting lemma.}
  \label{fig:densify}
\end{figure}
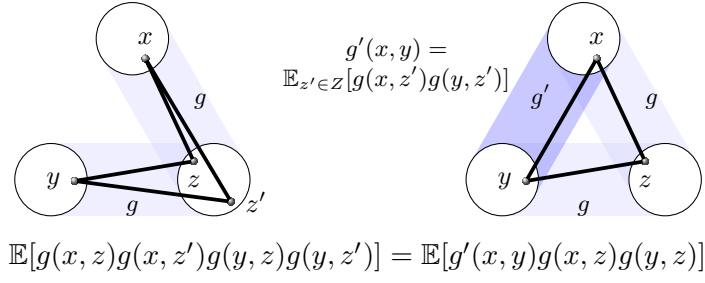

The proof uses repeated application of the Cauchy-Schwarz inequality,
a standard technique in this area, popularized by
Gowers~\cite{Gow98,Gow01,Gow06,Gow07}. The key additional idea, introduced in \cite{CFZrelsz, CFZ14adv}, is
\emph{densification} (see Figure~\ref{fig:densify}). After several applications of the Cauchy-Schwarz
inequality, it becomes necessary to analyze the $4$-cycle density:
$\EE_{x,y,z,z'}[g(x,z)g(x,z')g(y,z)g(y,z')]$. To do this, one introduces an auxiliary
weighted graph $g' \colon X \x Y \to [0,\infty)$ defined by $g'(x,y)
:= \EE_{z'}[g(x,z')g(y,z')]$ (this is basically the codegree function).
Note that we benefit here from working with weighted graphs.
The expression for the $4$-cycle density
now becomes $\EE_{x,y,z}[g'(x,y)g(x,z)g(y,z)]$.

At first glance, it seems that our reasoning is circular. Our aim was to estimate a certain triangle density expression but we have now returned to another triangle density expression. However, $g'$ behaves much more
like a dense weighted graph with bounded edge
weights, so what we have accomplished is to replace one of the ``sparse'' $g_{XY}$ by a
``dense'' $g'_{XY}$. If we do this two more times, replacing $g_{YZ}$ and $g_{XZ}$ with dense counterparts, the problem
reduces to the dense case, which we already know how to handle.

We begin with a warm-up showing how to apply the Cauchy-Schwarz
inequality (there will be many more applications later on). The following lemma
shows that the $3$-linear forms condition on $\nu$ implies $\norm{\nu -
  1}_\square = o(1)$, which we need to apply the dense model
theorem, Theorem~\ref{thm:dense-model}.

\begin{lemma} \label{lem:graph-nu-1}
  For any $\nu \colon X \x Y \to \RR$,
  \begin{equation} \label{eq:cut<=gowers}
  \norm{\nu - 1}_\square \leq (\EE_{x,x'\in X,y,y'\in
    Y}[(\nu(x,y)-1)(\nu(x',y) - 1)(\nu(x,y') - 1)(\nu(x',y')-1)])^{1/4}.
  \end{equation}
\end{lemma}

\begin{remark} The right-hand side of \eqref{eq:cut<=gowers} is the Gowers uniformity
  norm of $\nu - 1$. The lemma shows that the cut norm is weaker
  than the Gowers uniformity norm.
To see $\norm{\nu - 1}_\square = o(1)$, we expand the right-hand side
of \eqref{eq:cut<=gowers} into an alternating sum of linear forms in
$\nu$, each being $1 + o(1)$ by the linear forms condition, so that
the alternating sum cancels to $o(1)$.
\end{remark}

\begin{proof}
  By repeated applications of the Cauchy-Schwarz inequality, we have,
  for $A \subseteq X$ and $B \subseteq Y$,
  \begin{align*}
  \abs{\EE_{x,y}[(\nu(x,y) - 1)1_A(x)1_B(y)]}^4
  &\leq \abs{\EE_x[(\EE_y[(\nu(x,y) - 1)1_B(y)])^2 1_A(x)] }^2
  \\
  &\leq \abs{\EE_x[(\EE_y[(\nu(x,y) - 1)1_B(y)])^2] }^2
  \\
  &= \abs{\EE_{x,y,y'}[(\nu(x,y) - 1)(\nu(x,y')-1)1_B(y)1_B(y')]}^2
  \\
  &\leq \EE_{y,y'}[(\EE_x[(\nu(x,y) - 1)(\nu(x,y')-1)])^2 1_B(y)1_B(y')]
  \\
  &\leq \EE_{y,y'}[(\EE_x[(\nu(x,y) - 1)(\nu(x,y')-1)])^2]
  \\
  &= \EE_{x,x',y,y'}[(\nu(x,y)-1)(\nu(x',y) - 1)(\nu(x,y') - 1)(\nu(x',y')-1)].
  \end{align*}
  The lemma then follows.
\end{proof}

The next lemma is crucial to what follows. It shows that in certain expressions a factor $\nu$ can be
deleted from an expectation while incurring only a $o(1)$ loss.

\begin{lemma}[Strong linear forms] \label{lem:graph-slf}
  Let $\nu,g,\tg$ be weighted tripartite graphs on $X \cup Y \cup Z$.
  Assume that $\nu$ satisfies the $3$-linear forms condition, $0 \leq g \leq \nu$,
  and $0 \leq \tg \leq 1$. Then
  \[
  \EE_{x\in X, y \in Y, z,z'\in Z}[(\nu(x,y) -
  1)g(x,z)g(x,z')g(y,z)g(y,z')] = o(1)
  \]
  and the same statement holds if any subset of the four $g$ factors are
  replaced by $\tg$.
\end{lemma}

\begin{proof}
We give the proof when none of the $g$ factors are replaced. The other
  cases require only a simple modification. By the Cauchy-Schwarz
  inequality, we have
  \begin{align*}
  &\abs{\EE_{x,y,z,z'}[(\nu(x,y) -
  1)g(x,z)g(x,z')g(y,z)g(y,z')]}^2
  \\
  &\leq \EE_{y,z,z'}[(\EE_x[(\nu(x,y) -
  1)g(x,z)g(x,z')])^2g(y,z)g(y,z')] \ \EE_{y,z,z'}[g(y,z)g(y,z')]
  \\
    &\leq \EE_{y,z,z'}[(\EE_x[(\nu(x,y) -
  1)g(x,z)g(x,z')])^2\nu(y,z)\nu(y,z')] \ \EE_{y,z,z'}[\nu(y,z)\nu(y,z')].
  \end{align*}
  The second factor is at most $1 + o(1)$ by the linear forms
  condition. So it remains to analyze the first factor. We have, by
  another application of the Cauchy-Schwarz inequality,
  \begin{align*}
    &
    \abs{\EE_{y,z,z'}[(\EE_x[(\nu(x,y) -
    1)g(x,z)g(x,z')])^2\nu(y,z)\nu(y,z')]}^2
    \\
    & = \abs{\EE_{x,x',y,z,z'}[(\nu(x,y) -
    1)(\nu(x',y) - 1)g(x,z)g(x,z') g(x',z)g(x',z')\nu(y,z)\nu(y,z')]}^2
    \\
    & = \abs{\EE_{x,x',z,z'}[\EE_y[(\nu(x,y) -
    1)(\nu(x',y) - 1) \nu(y,z)\nu(y,z')]g(x,z)g(x,z')
    g(x',z)g(x',z')]}^2
    \\
    & \leq \EE_{x,x',z,z'}[(\EE_y[(\nu(x,y) -
    1)(\nu(x',y) - 1) \nu(y,z)\nu(y,z')])^2 g(x,z)g(x,z')
    g(x',z)g(x',z')] \\
    &\qquad \cdot \EE_{x,x',z,z'} [g(x,z)g(x,z')
    g(x',z)g(x',z')]
    \\
    & \leq \EE_{x,x',z,z'}[(\EE_y[(\nu(x,y) -
    1)(\nu(x',y) - 1) \nu(y,z)\nu(y,z')])^2 \nu(x,z)\nu(x,z')
    \nu(x',z)\nu(x',z')] \\
    &\qquad \cdot \EE_{x,x',z,z'} [\nu(x,z)\nu(x,z')
    \nu(x',z)\nu(x',z')].
  \end{align*}
  Using the $3$-linear forms condition, the second factor is $1 + o(1)$
  and the first factor is $o(1)$ (expand everything and observe that
  all the terms are $1+o(1)$ and the signs make all the $1$'s
  cancel).
\end{proof}

\begin{proof}[Proof of Theorem~\ref{thm:tri-count-sparse}]
  If $\nu$ is identically $1$, we are in the dense setting, in which case the theorem
  follows from Proposition~\ref{prop:tri-count-dense}. Now we apply
  induction on the number of $\nu_{XY}, \nu_{XZ}, \nu_{YZ}$ which are
  identically 1. By relabeling if necessary, we may assume without
  loss of generality that $\nu_{XY}$ is not identically 1. We define
  auxiliary weighted graphs $\nu', g', \tg' \colon X \x Y \to
  [0,\infty)$ by
  \begin{align*}
    \nu'(x,y) &:= \EE_z[\nu(x,z)\nu(y,z)], \\
    g'(x,y) &:= \EE_z[g(x,z)g(y,z)], \\
    \tg'(x,y) &:= \EE_z[\tg(x,z)\tg(y,z)].
  \end{align*}
  We refer to this step as \emph{densification}. The idea is that even
  though $\nu$ and $g$ are possibly unbounded, the new weighted
  graphs $\nu'$ and $g'$ behave like dense graphs. The weights on
  $\nu'$ and $g'$ are not necessarily bounded by $1$, but they almost
  are. We cap the weights by setting $g'_{\wedge 1} := \min\{g', 1\}$
  and $\nu'_{\wedge 1} := \min\{\nu', 1\}$ and show that
  the capping has negligible effect. We have
  \begin{equation} \label{eq:count-init-decomp}
    \EE[g(x,y)g(x,z)g(y,z) -
    \tg(x,y)\tg(x,z)\tg(y,z)]
    = \EE[g g' - \tg \tg']
    = \EE[g(g' - \tg')] + \EE[(g-\tg) \tg'],
  \end{equation}
  where the first expectation is taken over $x \in X, y \in Y, z \in
  Z$ and the other expectations are taken over $X \x Y$ (we will use
  these conventions unless otherwise specified).
The second term on the right-hand side of \eqref{eq:count-init-decomp}
equals $\EE[(g(x,y) - \tg(x,y))\tg(x,z)\tg(y,z)]$ and its absolute
value is at most $\norm{g - \tg}_\square = o(1)$ (here we
use $0 \leq \tg \leq 1$ as in the proof of
Proposition~\ref{prop:tri-count-dense}). So it remains to bound the
first term on the right-hand side of \eqref{eq:count-init-decomp}. By the Cauchy-Schwarz inequality, we have
\begin{align*}
(\EE[g(g' - \tg')])^2
\leq \EE[g(g' - \tg')^2] \ \EE[g]
&\leq \EE[\nu(g' - \tg')^2]\ \EE[\nu]
\\&=  \EE_{x,y}[\nu(x,y)(\EE_z[g(x,z)g(y,z) - \tg(x,z)\tg(y,z)])^2] \
\EE_{x,y}[\nu(x,y)].
\end{align*}
The second factor is $1 + o(1)$ by the linear forms condition. By Lemma~\ref{lem:graph-slf}, the
first factor differs from
\begin{equation} \label{eq:densify-sq}
\EE_{x,y}[(\EE_z[g(x,z)g(y,z) - \tg(x,z)\tg(y,z)])^2] = \EE[(g' - \tg')^2]
\end{equation}
by $o(1)$ (take the difference, expand the square, and then apply Lemma~\ref{lem:graph-slf} term-by-term).

The $3$-linear forms condition implies that $\EE[\nu'] = 1+o(1)$ and
$\EE[\nu'^2] = 1+o(1)$. Therefore, by the Cauchy-Schwarz inequality, we have
\begin{equation} \label{eq:nu'-1-sq}
(\EE[\sabs{\nu' - 1}])^2 \leq \EE[(\nu' - 1)^2] = o(1).
\end{equation}
We want to show
that \eqref{eq:densify-sq} is $o(1)$. We have
\begin{equation} \label{eq:cap-split}
\EE[(g' - \tg')^2]
= \EE[(g' - \tg')(g' - g'_{\wedge 1})] +
\EE[(g' - \tg')(g'_{\wedge 1} - \tg')].
\end{equation}
Since $0 \leq g' \leq \nu'$, we have
\begin{equation} \label{eq:cutoff-bound}
 0 \leq g' - g'_{\wedge 1} = \max\{g' - 1, 0\} \le \max\{\nu' - 1, 0\}
 \leq \sabs{\nu' - 1}.
\end{equation}
Using \eqref{eq:nu'-1-sq} and \eqref{eq:cutoff-bound}, the
absolute value of the first term on the right-hand side of
\eqref{eq:cap-split} is at most
\[
\EE[(\nu'+1)\sabs{\nu'-1}] = \EE[(\nu'-1)\sabs{\nu'-1}] + 2
\EE[\sabs{\nu'-1}] = o(1).
\]

Next, we claim that
\begin{equation}
  \label{eq:densified-cut}
\norm{g'_{\wedge 1} - \tg'}_\square = o(1).
\end{equation}
Indeed, for any $A \subseteq X$ and $B \subseteq Y$, we have
\[
\EE_{x,y}[(g'_{\wedge 1} - \tg')(x,y) 1_A(x)1_B(y)]
= \EE[(g'_{\wedge 1} - \tg') 1_{A \x B}]
= \EE[(g'_{\wedge 1} - g') 1_{A \x B}] + \EE[(g' - \tg') 1_{A \x B}].
\]
By \eqref{eq:cutoff-bound} and \eqref{eq:nu'-1-sq}, the absolute value
of the first term is at most $\EE[\abs{\nu' - 1}] = o(1)$. The
second term can be rewritten as
\[
\EE_{x,y,z}[1_{A\x B}(x,y) g(x,z)g(y,z) - 1_{A\x B}(x,y)\tg(x,z) \tg(y,z)],
\]
which is $o(1)$ by the induction hypothesis (replace $\nu_{XY}, g_{XY}, \tg_{XY}$
by $1$, $1_{A\x B}$, $1_{A\x B}$, respectively, and note that this increases the number of $\{\nu_{XY}, \nu_{XZ},
\nu_{YZ}\}$ which are identically $1$). This proves \eqref{eq:densified-cut}.

We now expand the second term on the right-hand side of
\eqref{eq:cap-split} as
\begin{equation} \label{eq:cap-split-2}
\EE[(g' - \tg')(g'_{\wedge 1} - \tg')]
= \EE[g' g'_{\wedge 1}] - \EE[g'\tg'] - \EE[\tg'g'_{\wedge1}] + \EE[\tg'^2].
\end{equation}
We claim that each of the expectations on the right-hand side is
$\EE[(\tg')^2] + o(1)$. Indeed, we have
\[
\EE[g' g'_{\wedge 1}] - \EE[(\tg')^2] = \EE_{x,y,z}[g'_{\wedge 1}(x,y)
g(x,z)g(y,z) - \tg'(x,y) \tg(x,z)\tg(y,z)],
\]
which is $o(1)$ by the induction hypothesis (replace
$\nu_{XY}, g_{XY}, \tg_{XY}$ by $1, g'_{\wedge1}, \tg'$, respectively,
which by \eqref{eq:densified-cut} satisfies $\norm{g'_{\wedge 1} - \tg'}_\square = o(1)$,
and note that this increases the number of $\{\nu_{XY}, \nu_{XZ},
\nu_{YZ}\}$ which are identically $1$). One can similarly show that the other
expectations on the right-hand side of \eqref{eq:cap-split-2} are also
$\EE[(\tg')^2] + o(1)$. Thus \eqref{eq:cap-split-2} is $o(1)$ and the
theorem follows.
\end{proof}


The main difficulty in extending Theorem~\ref{thm:tri-count-sparse} to hypergraphs is
notational. As discussed in Section \ref{sec:rel-sz}, to study $k$-APs, we
consider $(k-1)$-uniform $k$-partite weighted hypergraphs. The
vertex sets will be denoted $X_1, \dots, X_k$ (in application $X_i =
\ZZ_N$ for all $i$). We write $X_{-i} := X_1 \x \cdots \x X_{i-1} \x
X_{i+1} \x \cdots \x X_k$ and $x_{-i} := (x_1, \dots, x_{i-1}, x_{i+1},
\dots, x_k)$ for any $x = (x_1, \dots, x_k) \in X_1 \x
\cdots \x X_k$. Then a weighted hypergraph $g$ consists of functions
$g_{-i} \colon X_{-i} \to
\RR$ for each $i = 1, \dots, k$. As before, we drop the subscripts if
they are clear from context. We write $\norm{g}_\square =
\max\{\norm{g_{-1}}_\square, \dots, \norm{g_{-k}}_\square\}$, where $\norm{g_{-i}}_\square$ is
the cut norm of $g_{-i}$ defined in \eqref{eq:hyp-cut-norm}.

The appropriate generalization of the $3$-linear forms condition involves counts for the $2$-blow-up of the simplex $K_k^{(k-1)}$. We say that a weighted
hypergraph $\nu$ satisfies the \emph{$k$-linear forms condition} (the
hypergraph version of Definition~\ref{def:Z-lfc}) if
\begin{equation*}
  \label{eq:hyp-lfc}
  \EE_{x_1^{(0)}, x_1^{(1)} \in X_1, \dots, x_k^{(0)},x_k^{(1)} \in X_k} \Bigl[
  \prod_{j=1}^k \prod_{\omega \in \{0,1\}^{[k]\setminus\{j\}}} \nu(x_{-j}^{(\omega)})
  \Bigr]  = 1 + o(1)
\end{equation*}
and also the same statement holds if any subset of the
$\nu$ factors (there are $k2^{k-1}$ such factors) are deleted. Here $x_{-j}^{(\omega)} :=
(x_1^{(\omega_1)}, \dots, x_{j-1}^{(\omega_{j-1})},
x_{j+1}^{(\omega_{j+1})}, \dots, x_{k}^{(\omega_{k})}) \in
X_{-j}$.

The following theorem
generalizes Theorem~\ref{thm:tri-count-sparse}.

\begin{theorem}[Relative simplex counting lemma]
  \label{thm:hyp-count}
  Let $\nu, g, \tg$ be weighted $(k-1)$-uniform $k$-partite weighted
  hypergraphs on $X_1 \cup \cdots \cup X_k$. Assume that $\nu$
  satisfies the $k$-linear forms condition, $0 \leq g \leq \nu$ and $0
  \leq \tg \leq 1$. If $\norm{g - \tg}_\square = o(1)$, then
  \[
  \abs{\EE_{x_1 \in X_1, \dots, x_k \in X_k} [g(x_{-1}) g(x_{-2})
    \cdots g(x_{-k}) - \tg(x_{-1}) \tg(x_{-2})
    \cdots \tg(x_{-k})]} = o(1).
  \]
\end{theorem}

The proof of Theorem~\ref{thm:hyp-count} is a straightforward generalization of
the proof of Theorem~\ref{thm:tri-count-sparse}. We simply point out the necessary
modifications and leave the reader to figure out the details (a full proof can be found
in our paper \cite{CFZrelsz}, but it is perhaps easier to reread the graph case and think about the small changes that need to
be made).

The proof proceeds by induction on the number of $\nu_{-1}, \dots,
\nu_{-k}$ which are not identically 1. When $\nu = 1$, we are in the
dense setting and the proof of Proposition~\ref{prop:tri-count-dense}
easily extends. Now assume that $\nu_{-1}$ is not identically $1$. We
have extensions of Lemmas \ref{lem:graph-nu-1} and
\ref{lem:graph-slf}, where in the proof we have to apply the
Cauchy-Schwarz inequality $k-1$ times in succession. For the
densification step, we define $\nu', g', \tg' \colon X_{-1} \to [0,\infty)$ by
\begin{align*}
\nu'(x_{-1}) &= \EE_{x_1 \in X_1}[ \nu(x_{-2}) \cdots
\nu(x_{-k})], \\
g'(x_{-1}) &= \EE_{x_1 \in X_1}[ g(x_{-2}) \cdots
g(x_{-k})], \\
\tg'(x_{-1}) &= \EE_{x_1 \in X_1}[ \tg(x_{-2}) \cdots
\tg(x_{-k})].
\end{align*}
The rest of the proof works with minimal changes.

\section{Proof of the relative Szemer\'edi theorem} \label{sec:pf-rel-sz}

We are now ready to prove the relative Szemer\'edi theorem using the dense model theorem and the
counting lemma following the outline given in Section \ref{sec:rel-sz}.

\begin{proof}[Proof of Theorem~\ref{thm:rel-sz}]
  The $k$-linear forms condition implies that $\norm{\nu -
    1}_{\square, k-1} = o(1)$ (by a sequence of $k-1$ applications of
  the Cauchy-Schwarz inequality, following
  Lemma~\ref{lem:graph-nu-1}).  By the dense model theorem,
  Theorem~\ref{thm:dense-model}, we can find $\tf \colon \ZZ_N \to
  [0,1]$ so that $\snorm{f - \tf}_{\square,k-1} = o(1)$.

  Let $X_1 = X_2 = \cdots = X_k = \ZZ_N$. For each $j = 1, \dots, k$,
  define the linear form $\psi_j \colon X_{-j} \to \ZZ_N$ by
  \[
  \psi_j (x_1, \dots, x_{j-1}, x_{j+1}, \dots, x_k) := \sum_{i \in
    [k]\setminus\{j\}}(j-i) x_i.
  \]
  Construct $(k-1)$-uniform $k$-partite weighted hypergraphs $\nu, g,
  \tg$ on $X_1 \cup \cdots \cup X_k$ by setting
  \[
  \nu_{-j}(x_{-j}) := \nu(\psi_j(x_{-j})), \quad
  g_{-j}(x_{-j}) := f(\psi_j(x_{-j})), \quad
  \tg_{-j}(x_{-j}) := \tf(\psi_j(x_{-j})) \quad
  \]
  (in the first definition, the left $\nu_{-j}$ refers to the weighted
  hypergraph and the second $\nu$ refers to the given function on
  $\ZZ_N$). We claim that
  \begin{equation} \label{eq:nu-cut-Z-hyp}
  \snorm{\nu_{-j} - 1}_\square = \snorm{\nu - 1}_{\square, k-1}
  \end{equation}
  and
  \begin{equation} \label{eq:g-f-cut-Z-hyp}
  \snorm{g_{-j} - \tg_{-j}}_\square = \snorm{f - \tf}_{\square, k-1}
  \end{equation}
  for every $j$ (in both \eqref{eq:nu-cut-Z-hyp} and
  \eqref{eq:g-f-cut-Z-hyp} the left-hand side refers to the hypergraph
  cut norm \eqref{eq:hyp-cut-norm} while the right-hand side refers to
  the cut norm \eqref{eq:Z-cut-norm} for functions on $\ZZ_N$). We
  illustrate \eqref{eq:g-f-cut-Z-hyp} in the case when $k=j=4$ (the
  full proof is straightforward). The
  left-hand side of \eqref{eq:g-f-cut-Z-hyp} equals
  \begin{equation}\label{eq:3-cut-norm-a}
    \sup_{A_1, A_2, A_3 \subseteq \ZZ_N^2} \abs{\EE_{x_1, x_2, x_3 \in
        \ZZ_N}[(f - \tf)(3x_1 + 2x_2 + x_3)
      1_{A_1}(x_2,x_3)1_{A_2}(x_1,x_3) 1_{A_3}(x_1,x_2)]},
  \end{equation}
  while the right-hand side of \eqref{eq:g-f-cut-Z-hyp} equals
  \begin{equation} \label{eq:3-cut-norm-b}
    \sup_{B_1, B_2, B_3 \subseteq \ZZ_N^2} \abs{\EE_{x_1, x_2, x_3 \in
        \ZZ_N}[(f - \tf)(x_1 + x_2 + x_3)
      1_{B_1}(x_2,x_3)1_{B_2}(x_1,x_3) 1_{B_3}(x_1,x_2)]}.
  \end{equation}
  These two expressions are equal\footnote{Here we use the
    assumption in the footnote to Definition~\ref{def:Z-lfc} that $N$
    is coprime to $(k-1)!$. Without this assumption, it can be shown
    that the two norms~\eqref{eq:3-cut-norm-a} and
    \eqref{eq:3-cut-norm-b} differ by at most a constant factor
    depending on $k$, which would also suffice for what follows.} up to a change
  of variables $3x_1 \leftrightarrow x_1$ and $2x_2 \leftrightarrow
  x_2$.

  It follows from \eqref{eq:g-f-cut-Z-hyp} that $\snorm{g -
    \tg}_\square = \snorm{f - \tf}_{\square, k-1} = o(1)$. Moreover, the
  $k$-linear forms condition for $\nu \colon \ZZ_N \to [0,\infty)$
  translates to the $k$-linear forms condition for the weighted
  hypergraph $\nu$. It follows from the counting lemma,
  Theorem~\ref{thm:hyp-count}, that
  \begin{equation} \label{eq:appy-count-lem}
    \EE_{x_1, \dots, x_k \in \ZZ_N^k}[g_{-1}(x_{-1}) \cdots
    g_{-k}(x_{-k})] = \EE_{x_1, \dots, x_k \in \ZZ_N^k}[\tg_{-1}(x_{-1})
    \cdots \tg_{-k}(x_{-k})] + o(1).
  \end{equation}
  The left-hand side is equal to
  \[
  \EE_{x_1, \dots, x_k \in \ZZ_N^k}[f(\psi_1(x_{-1})) \cdots
  f(\psi_k(x_{-k}))]
  = \EE_{x,d\in\ZZ_N}[f(x)f(x+d) \cdots f(x+(k-1) d)],
  \]
  which can be seen by setting $x = \psi_1(x_{-1})$ and $d = x_1 +
  \cdots + x_k$ so that $\psi_j (x_{-j}) = x + (j-1)d$. A similar
  statement holds for the right-hand side of
  \eqref{eq:appy-count-lem}. So \eqref{eq:appy-count-lem} is
  equivalent to
  \[
  \EE_{x,d\in\ZZ_N}[f(x)f(x+d) \cdots f(x+(k-1) d)]
  = \EE_{x,d\in\ZZ_N}[\tf(x)\tf(x+d) \cdots \tf(x+(k-1) d)] + o(1),
  \]
  which is at least $c(k,\delta) - o_{k,\delta}(1)$ by
  Theorem~\ref{thm:sz-weighted}, as desired.
\end{proof}

\section{Constructing the majorant} \label{sec:maj}

In this section, we use the relative Szemer\'edi theorem
to prove the Green-Tao theorem. To do this, we must construct a
majorizing measure for the primes that satisfies the linear
forms condition.

Rather than considering the set of primes itself, we put weights on the
primes, a common technique in analytic number theory. The weights we
use will be related to the well-known von Mangoldt function $\Lambda$.
This is defined by $\Lambda(n) = \log p$ if $n
= p^k$ for some prime $p$ and positive integer $k$ and $\Lambda(n) = 0$ if $n$
is not a power of a prime (actually, the higher powers $p^2$, $p^3,
\dots$ play no role here and we will soon discard them from $\Lambda$).
That these are natural weights to consider follows from the
observation that the Prime Number Theorem is
equivalent to $\sum_{n \leq N} \Lambda(n) = (1 + o(1))N$.

A difficulty with using $\Lambda$ is that it is biased on certain residue
classes. For example, every prime other than $2$ is odd. This
prevents us from making any pseudorandomness claims unless we can
somehow remove these biases. This is achieved using the
\emph{W-trick}. Let $w = w(N)$ be any function that tends to infinity
slowly with $N$. Let $W = \prod_{p \leq w} p$ be the product of primes
up to $w$. The trick for avoiding biases mod $p$ for any $p \leq w$ is to consider only those
primes which are congruent to $1$ (mod $W$). In keeping with this idea, we define the modified von Mangoldt
function by
\[
\wt \Lambda(n) := \begin{cases}
  \frac{\phi(W)}{W} \log (W n + 1) & \text{when $Wn + 1$ is prime}, \\
  0 & \text{otherwise}.
\end{cases}
\]
The factor $\phi(W)/W$ is present since exactly $\phi(W)$ of the $W$
residue classes mod $W$ have infinitely many primes and a strong form of Dirichlet's
theorem\footnote{In fact,
Dirichlet's theorem, or even the Prime Number Theorem, are not
necessary to prove the Green-Tao theorem, though we assume them to simplify the exposition. Indeed, a
weaker form of the Prime Number Theorem asserting that there are at
least $c N/\log N$ primes up to $N$ for some $c > 0$ suffices for our
needs (this bound was first proved by Chebyshev and a famous short
proof was subsequently found by Erd\H{o}s; see
\cite[Ch.~2]{AZ10}). Furthermore, in place of Dirichlet's theorem, a
simple pigeonhole argument shows that for each $W$, some residue class
$b \pmod W$ contains many primes (whereas we use Dirichlet's theorem to take
$b=1$). The proof presented here can easily be modified to deal with
general $b$, though the notation gets a bit more cumbersome as $b$
could vary with $W$. An analysis of this sort is necessary to prove a Szemer\'edi-type statement for the primes (see Section \ref{sec:conclusion}),
since we do not then know how our subset of the primes
is distributed on congruence classes.} tells us that the primes are equidistributed among these
$\phi(W)$ residue classes, i.e., $\sum_{n \leq N} \wt \Lambda(n) = (1
+ o(1))N$ as long as $w$ grows slowly enough with $N$. From now on, we
will work with $\wt \Lambda$ rather than $\Lambda$. Our main goal is to
prove the following result, which says that there is a majorizing measure for $\wt \Lambda$ which
satisfies the linear forms condition.

\begin{proposition}
  \label{prop:majorize-Lambda}
  For every $k \geq 3$, there exists $\delta_k > 0$ such that for every sufficiently large $N$ there exists a function $\nu \colon
  \ZZ_N \to [0,\infty)$ satisfying the $k$-linear forms condition and
  $\nu(n) \geq \delta_k \wt \Lambda(n)$ for all $N/2 \leq n
  < N$.
\end{proposition}

Using this majorant with the relative Szemer\'edi theorem, we obtain the Green-Tao theorem.

\begin{proof}[Proof of Theorem~\ref{thm:gt} assuming
  Proposition~\ref{prop:majorize-Lambda}]
  Define $f \colon \ZZ_N \to [0,\infty)$ by $f(n) = \delta_k \wt
  \Lambda(n)$ if $N/2 \leq n < N$ and $f(n) = 0$ otherwise. By
  Dirichlet's theorem, $\sum_{N/2 \leq n < N} f(n) = (1/2 +
  o(1)) \delta_k N$, so $\EE f \geq \delta_k /3$ for large $N$. Since
  $0 \leq f \leq \nu$ and $\nu$ satisfies the $k$-linear forms
  condition, it follows from the relative Szemer\'edi theorem,
  Theorem~\ref{thm:rel-sz}, that $\EE[f(x) f(x+d) \cdots f(x+(k-1)d)]
  \geq c(k, \delta_k/3) - o_{k,\delta}(1)$. Therefore, for sufficiently large $N$, we
  have $f(x) f(x+d) \cdots f(x+(k-1)d) > 0$ for some $N/2 \leq x < N$
  and $d \neq 0$ (the $d=0$ terms contribute negligibly to the
  expectation). Since $f$ is supported on $[N/2,N)$, we see that $x,
  x+d, \dots, x+(k-1)d$ is not only an AP in $\ZZ_N$ but also in
  $\ZZ$, i.e., has no wraparound issues. Thus $(x+jd)W+1$ for $j=0,\ldots,k-1$ is a $k$-AP of primes.
\end{proof}

How do we construct the majorant $\nu$ for $\wt \Lambda(n)$? Recall
that the M\"obius function $\mu$ is defined by $\mu(n) =
(-1)^{\omega(n)}$ when $n$ is square-free, where $\omega(n)$ is the
number of prime factors of $n$, and $\mu(n) = 0$ when $n$ is not
square-free.  The functions $\Lambda$ and $\mu$ are related by the
M\"obius inversion formula
\[
\Lambda(n) = \sum_{d \mid n} \mu(d) \log(n/d).
\]
In Green and Tao's original proof, the following truncated
version of $\Lambda$ (motivated by \cite{GY03}) was used to construct the majorant. For any $R > 0$, define
\[
\Lambda_R(n) := \sum_{\substack{d \mid n \\ d \leq R}} \mu(d) \log(R/ d).
\]
Observe that if $n$ has no prime divisors less than or equal to $R$, then
$\Lambda_R(n) = \log R$. Tao~\cite{Taonote} later simplified the proof by using the
following variant of $\Lambda_R$, where the restriction $d \leq R$
is replaced by a smoother cutoff.

\begin{definition}
  Let $\chi \colon \RR \to [0,1]$ be any smooth, compactly supported
  function. Define
  \[
  \Lambda_{\chi,R}(n) := \log R \sum_{d \mid n} \mu(d)
  \chi\paren{\frac{\log d}{\log R}}.
  \]
\end{definition}
In our application, $\chi$ will be supported on $[-1,1]$, so only
divisors $d$ which are at most $R$ are considered in the sum.  Note that $\Lambda_R$
above corresponds to $\chi(x) = \max\{1 - \sabs{x}, 0\}$, which is
not smooth. The following proposition, which we will prove in the next
section, gives a linear forms estimate for $\Lambda_{\chi, R}$.

\begin{proposition}[Linear forms estimate] \label{prop:Lambda-lf} Fix
  any smooth function $\chi \colon \RR \to
  [0,1]$ supported on $[-1,1]$. Let $m$ and $t$ be positive integers. Let $\psi_1, \dots,
  \psi_m \colon \ZZ^t \to \ZZ$ be fixed linear maps, with no two being
  multiples of each other. Assume that $R = o(N^{1/(10m)})$ grows
  with $N$ and $w$
  grows sufficiently slowly with $N$. Let $W := \prod_{p \leq w}
  p$. Write $\theta_i := W \psi_i + 1$. Let $B$ be a product
  $\prod_{i=1}^t I_i$, where each $I_i$ is a set of at least $R^{10m}$
  consecutive integers. Then
    \begin{equation} \label{eq:Lambda-lfc}
    \EE_{\bx \in B}[\Lambda_{\chi,R}(\theta_1(\bx))^2 \cdots
      \Lambda_{\chi,R}(\theta_m(\bx))^2]
    = (1 + o(1)) \paren{\frac{W c_\chi \log R}{\phi(W) }}^m,
  \end{equation}
  where $o(1)$ denotes a quantity tending to zero as $N \to \infty$
  (at a rate that may depend on $\chi$, $m$, $t$, $\psi_1, \dots, \psi_m$, $R$,
  and $w$), and
  $c_\chi$ is the normalizing factor
  \[
  c_\chi := \int_0^\infty \sabs{\chi'(x)}^2 \, dx.
  \]
\end{proposition}

Now we construct the majorizing measure $\nu$ and show that it satisfies the linear
forms condition.

\begin{proposition}
  \label{prop:nu-lf}
  Fix any smooth function $\chi \colon \RR \to [0,1]$ supported on
  $[-1,1]$ with $\chi(0) = 1$.  Let $k \geq 3$ and $R := N^{k^{-1}
    2^{-k-3}}$. Assume that $w$ grows sufficiently slowly with $N$ and
  let $W := \prod_{p \leq w} p$. Define $\nu \colon \ZZ_N \to
  [0,\infty)$ by
  \begin{equation}\label{eq:construct-nu}
  \nu(n) := \begin{cases}
    \frac{\phi(W)}{W} \frac{\Lambda_{\chi, R}(Wn+1)^2}{c_\chi \log R} &
    \text{when } N/2 \leq n < N, \\
    1 & \text{otherwise.}
  \end{cases}
\end{equation}
Then $\nu$ satisfies the $k$-linear forms condition.
\end{proposition}

Note that while $\Lambda_{\chi,R}$ is not necessarily nonnegative,
$\nu$ constructed in \eqref{eq:construct-nu} is always nonnegative due to
the square on $\Lambda_{\chi,R}$.

\begin{proof}[Proof of Proposition~\ref{prop:majorize-Lambda} assuming
  Proposition~\ref{prop:nu-lf}]
  Take $\delta_k = k^{-1}2^{-k-4}c_\chi^{-1}$. It suffices to verify
  that for $N$ sufficiently large
  we have $\delta_k \wt \Lambda(n) \leq \nu(n)$ for all $N/2 \leq n
  < N$. We only need to check
  the inequality when $Wn + 1$ is prime, since $\wt \Lambda(n)$ is
  zero otherwise. We have
  \[
  \log R = k^{-1}2^{-k-3} \log N \geq k^{-1}2^{-k-4}\log(WN + 1) =
  c_\chi\delta_k \log(WN + 1),
  \]
  where the inequality holds for sufficiently large $N$ provided
  $w$ grows slowly enough. When $Wn + 1$ is prime, we have $\Lambda_{\chi,R}(Wn + 1) = \log
  R$, so
  \[
  \delta_k \wt \Lambda(n)
  =\delta_k \frac{\phi(W)}{W} \log(Wn + 1)
  \leq \delta_k \frac{\phi(W)}{W} \log(WN + 1)
  \leq \frac{\phi(W)}{W} \frac{\log R}{c_\chi} = \nu(n),
  \]
  as claimed.\end{proof}

\begin{proof}[Proof of Proposition~\ref{prop:nu-lf} assuming
  Proposition~\ref{prop:Lambda-lf}]
  We need to check that
  \begin{equation} \label{eq:verify-nu}
  \EE_{\bx \in
    \ZZ_N^t}[\nu(\psi_1(\bx)) \cdots \nu(\psi_m(\bx))] = 1 + o(1)
\end{equation}
whenever $\psi_1, \dots, \psi_m$, $m \leq k 2^{k-1}$, are the linear
forms that appear in \eqref{eq:Z-lfc} or any subset thereof. Note that no two
$\psi_i$ are multiples of each other.

To use the two-piece definition of $\nu$, we divide the
  domain $\ZZ_N$ into intervals. Let $Q = Q(N)$ be a slowly
  increasing function of $N$. Divide $\ZZ_N$ into $Q$ roughly equal
  intervals and form a partition of $\ZZ_N^t$ into $Q^t$ boxes, as follows:
  \[
  B_{u_1, \dots, u_t} = \prod_{j=1}^t ([u_j N/Q, (u_j + 1)N/Q) \cap
  \ZZ_N) \subseteq \ZZ_N^t, \qquad u_1, \dots, u_t \in \ZZ_Q.
  \]
  Then, up to a $o(1)$ error (due to the fact that the boxes do not
  all have exactly equal sizes), the left-hand side of \eqref{eq:verify-nu} equals
  \[
  \EE_{u_1, \dots,
        u_t \in \ZZ_Q}[\EE_{\bx \in B_{u_1, \dots, u_t}}[\nu(\psi_1(\bx)) \cdots
        \nu(\psi_m(\bx))]].
  \]
  We say that a box $B_{u_1, \dots, u_t}$ is \emph{good} if, for each
  $j \in [m]$, the set $\{\psi_j(\bx) : \bx \in B_{u_1, \dots, u_t}\}$
  either lies completely in the subset $[N/2, N)$ of
  $\ZZ_N$ or completely outside this subset. Otherwise, we say that the box
  is \emph{bad}. We may assume $Q$ grows slowly enough
  that $N/Q \ge R^{10m}$. From Proposition~\ref{prop:Lambda-lf} and
  the definition of $\nu$, we know that for good boxes,
    \[
    \EE_{\bx \in B_{u_1, \dots, u_t}}[\nu(\psi_1(\bx)) \cdots
        \nu(\psi_m(\bx))] = 1+ o(1).
    \]
    For bad boxes, we use the bound $\nu(n) \leq
    1 + \frac{\phi(W)}{W} \frac{\Lambda_{\chi, R}(Wn+1)^2}{c_\chi
      \log R}$. By expanding and applying
    \eqref{eq:Lambda-lfc} to each term, we find that
    \[
    \EE_{\bx \in B_{u_1, \dots, u_t}}[\nu(\psi_1(\bx)) \cdots
        \nu(\psi_m(\bx))] = O(1)
    \]
    (it is bounded in absolute value by $2^m + o(1)$). It remains to
    show that the proportion of boxes that are bad is $o(1)$.

    Suppose $B_{u_1, \dots, u_t}$ is bad. Then there exists some $i$
    such that the image of the box under $\psi_i$ intersects both $[N/2,
    N)$ and its complement. This implies that there exists some
    (real-valued) $\bx \in \prod_{j=1}^t [u_j N / Q, (u_j+1)N /Q)
    \subseteq (\RR/N\ZZ)^t$ with $\psi_i(\bx) = 0$ or $N/2$ (mod $N$).
    Letting $\by = Q \bx / N$, we see that $\by \in \prod_{j=1}^t [u_j,
    u_j+1) \subseteq (\RR/Q\ZZ)^t$ satisfies $\psi_i(\by) = 0$ or
    $Q/2$ (mod $Q$).  This implies that $\psi_i(u_1, \dots, u_t)$ is
    either $O(1)$ or $Q/2 + O(1)$ (mod $Q$).  Since $\psi_i$ is a
    nonzero linear form, at most a $O(1/Q)$ fraction of the tuples
    $(u_1, \dots, u_t) \in \ZZ_Q^t$ have this property. This can be
    seen by noting that if we fix all but one of the coordinates, there will
    be $O(1)$ choices for the final coordinate for which $\psi_i$ is in the required range.
    Taking the union over all $i$, we see that the proportion of bad boxes is
    $O(1/Q) = o(1)$.
\end{proof}

\section{Verifying the linear forms condition} \label{sec:verify-lfc}

In this section, we prove Proposition~\ref{prop:Lambda-lf}. There are
numerous estimates along the way. To avoid getting bogged
down with the rather technical error bounds, we first go through
the proof while skipping some of these details (i.e., by only
considering the ``main term''). The approximations
are then justified at the end, where we collect the error bound arguments.
We note that all constants will depend implicitly on $\chi, m, t, \psi_1, \dots, \psi_m$.

Expanding the definition of $\Lambda_{\chi,R}$, we rewrite the left-hand side
of~\eqref{eq:Lambda-lfc} as
\begin{equation} \label{eq:expand-Lambda}
(\log R)^{2m} \sum_{d_1,d'_1, \dots, d_m, d'_m
    \in \NN} \left(\prod_{j=1}^m \mu(d_j)
  \chi\paren{\frac{\log d_j}{\log R}}
  \mu(d'_j) \chi\paren{\frac{\log d'_j}{\log R}}
  \right) \EE_{\bx \in B}[1_{d_j,d'_j
      \mid \theta_j(\bx) \ \forall j}].
\end{equation}
Since $\mu(d) = 0$ unless $d$ is square-free, we only need to consider
square-free $d_1,d'_1, \dots,d_m,d'_m$. Also, since $\chi$ is supported on $[-1,1]$, we may assume that $d_1, d'_1, \dots, d_m, d'_m \leq R$. Let
$D$ denote the lcm of $d_1, d'_1, \dots, d_m, d'_m$. The width of the box $B$ is at least $R^{10m}$ in each
dimension, so, by considering a slightly smaller box $B' \subseteq B$ such
that each dimension of $B'$ is divisible by $D \leq R^{2m}$, we obtain
\[
\EE_{\bx \in B}[1_{d_j,d'_j
      \mid \theta_j(\bx) \, \forall j}]
 =
\EE_{\bx \in \ZZ_D^t}[1_{d_j,d'_j
      \mid \theta_j(\bx) \, \forall j}]
+ O(R^{-8m}).
\]
Therefore, as there are at most $R^{2m}$ choices for $d_1, d'_1, \dots, d_m, d'_m$, we see that, up to an additive error of $O(R^{-6m}\log^{2m} R)$, we may approximate \eqref{eq:expand-Lambda} by
\begin{equation} \label{eq:pre-Fourier}
(\log R)^{2m} \sum_{d_1,d'_1, \dots, d_m, d'_m
    \in \NN} \left(\prod_{j=1}^m \mu(d_j)
  \chi\paren{\frac{\log d_j}{\log R}}
  \mu(d'_j) \chi\paren{\frac{\log d'_j}{\log R}}
  \right) \EE_{\bx \in \ZZ_D^t}[1_{d_j,d'_j
      \mid \theta_j(\bx) \ \forall j}] .
\end{equation}

Let $\varphi$ be the Fourier transform of $e^x\chi(x)$. That is,
\[
e^x \chi(x) = \int_\RR \varphi(\xi)e^{-ix\xi} \,d \xi.
\]
Substituting and simplifying, we have
\[
\chi\paren{\frac{\log d}{\log R}} = \int_\RR d^{-\frac{1
    + i\xi}{\log R}} \varphi(\xi) \, d\xi.
\]
We wish to plug this integral into \eqref{eq:pre-Fourier}. It helps to
first restrict the integral to a compact interval $I = [-\log^{1/2} R,
\log^{1/2} R]$. By basic results in Fourier analysis (see, for example,~\cite[Chapter 5, Theorem 1.3]{SSt03}), since $\chi$ is smooth and
compactly supported, $\varphi$ decays rapidly, that is, $\varphi(\xi) = O_A ((1
+ \sabs{\xi})^{-A})$ for any $A > 0$. It follows that for any $A > 0$,
\begin{equation}\label{eq:chi-int-trunc}
\chi\paren{\frac{\log d}{\log R}}
=
\int_I d^{-\frac{1
    + i\xi}{\log R}} \varphi(\xi) \, d\xi + O_A(d^{-1/\log R}
(\log R)^{-A}).
\end{equation}
We write
\[
z_j := \frac{1 + i\xi_j}{\log R}
\qquad\text{and}\qquad
z'_j := \frac{1 + i\xi'_j}{\log R}.
\]
We have $\chi(\log d / \log R) = O(d^{-1/\log R})$ (we only need to
check this for $d \leq R$ since $\chi$ is supported on
$[-1,1]$). Using \eqref{eq:chi-int-trunc}, we have
\begin{multline}
  \label{eq:int-trunc-prod}
\prod_{j=1}^m \chi\paren{\frac{\log d_j}{\log R}}\chi\paren{\frac{\log
    d'_j}{\log R}} =
\int_I \cdots \int_I \prod_{j=1}^m d_j^{-z_j} {d'_j}^{-z'_j}
\varphi(\xi_j) \varphi(\xi'_j) \, d\xi_j d\xi'_j \\ +   O_A\paren{ (\log R)^{-A} \prod_{j=1}^m (d_jd'_j)^{-1/\log R}}.
\end{multline}
Using \eqref{eq:int-trunc-prod}, we estimate \eqref{eq:pre-Fourier}  (error bounds are deferred
to the end) by
\begin{equation} \label{eq:post-swap}
(\log R)^{2m} \int_I \cdots \int_I
\sum_{d_1, d'_1, \dots, d_m, d'_m \in \NN}
\EE_{\bx \in \ZZ_D^t}[1_{d_j,d'_j
      \mid \theta_j(\bx) \ \forall j}]
\prod_{j=1}^m \mu(d_j) d_j^{-z_j} \mu(d'_j) {d'_j}^{-z'_j}
\varphi(\xi_j)\varphi(\xi'_j) \,d\xi_j d\xi'_j.
\end{equation}
We are allowed to swap the summation and the integrals because $I$ is
compact and the sum can be shown to be absolutely convergent
(the argument for absolute convergence is similar to the error bound
for \eqref{eq:post-swap} included towards the end of the section). Splitting $d_1,
d'_1, \dots, d_m, d'_m$ in \eqref{eq:post-swap} into prime
factors, we obtain
\begin{equation} \label{eq:euler-product}
  \eqref{eq:post-swap} = (\log
  R)^{2m} \int_I \cdots \int_I \prod_p E_p(\xi) \cdot \prod_{j=1}^m
  \varphi(\xi_j)\varphi(\xi'_j) \, d\xi_j d\xi'_j,
\end{equation}
where $\xi = (\xi_1, \xi'_1, \dots, \xi_m, \xi'_m) \in I^{2m}$ and
$E_p(\xi)$ is the Euler factor
\[
E_p(\xi) := \sum_{d_1, d'_1, \dots, d_m, d'_m \in \{1,p\}} \EE_{\bx \in \ZZ_p^t}[1_{d_j,d'_j
      \mid \theta_j(\bx) \ \forall j}]
\prod_{j=1}^m \mu(d_j) d_j^{-z_j} \mu(d'_j) {d'_j}^{-z'_j}.
\]
We have $E_p(\xi) = 1$ when $p \leq w$ (recall the $W$-trick, so $p
\nmid \theta_j(\bx) = W \psi_j(\bx) + 1$ for all $j$ when $p \leq
w$). When $p > w$, the expectation in the summand equals $1$ if all
$d_j, d'_j$ are $1$, $1/p$ if $d_j d'_j = 1$ for all
except exactly one $j$, and is at most $1/p^2$ otherwise (here we assume
that $w$ is sufficiently large so that no two $\psi_i$ are
multiples of each other mod $p$). It follows that
for $p > w$,
\[
E_p(\xi) = 1 - p^{-1} \sum_{j=1}^m(p^{-z_j} + p^{-z'_j} - p^{-z_j -
  z'_j}) + O(p^{-2})
         = (1 + O(p^{-2})) E'_p(\xi),
\]
where, for any prime $p$,
\[
E'_p(\xi) := \prod_{j=1}^m \frac{(1 - p^{-1-z_j})(1- p^{-1-z'_j})}{1 - p^{-1-z_j-z'_j}}.
\]
It then follows that
\begin{align} \label{eq:euler-compare}
  \prod_p E_p(\xi) = \prod_{p > w}(1 + O(p^{-2})) E'_p(\xi) = (1 +
  O(w^{-1})) \paren{\prod_{p\leq w} E'_p(\xi)}^{-1} \prod_p E'_p(\xi).
\end{align}
Recall that the Riemann zeta function
\[
\zeta(s) := \sum_{n \geq 1} n^{-s} = \prod_p (1 - p^{-s})^{-1}
\]
has a simple pole at $s=1$ with
residue $1$ (a proof is included towards the end). This implies that
\begin{equation} \label{E'_p-zeta}
  \prod_p E'_p(\xi) = \prod_{j=1}^m \frac{\zeta(1 + z_j +
    z'_j)}{\zeta(1 + z_j)\zeta(1 + z'_j)}
  \approx \prod_{j=1}^m \frac{z_j z'_j}{z_j + z'_j},
\end{equation}
where $\approx$ denotes asymptotic equality.
Here we use $\abs{z_1}, \abs{z'_1}, \dots, |z_m|, |z'_m| = O((\log R)^{-1/2})$ as
$\xi_1, \xi'_1, \dots \xi_m, \xi'_m \in I$. For $p \leq w$, we make the
approximation $E'_p(\xi) \approx (1 - p^{-1})^m$. Hence,
\begin{equation}\label{eq:E'_p-phi}
\prod_{p \leq w} E'_p(\xi) \approx \prod_{p \leq w} (1 - p^{-1})^m = \paren{\frac{\phi(W)}{W}}^m.
\end{equation}
Substituting \eqref{eq:euler-compare}, \eqref{E'_p-zeta}, and
\eqref{eq:E'_p-phi} into \eqref{eq:euler-product}, we find that
\begin{equation}
  \label{eq:pre-c_chi-eval}
  \eqref{eq:euler-product}
  \approx
  (\log R)^{2m} \paren{\frac{W}{\phi(W)}}^m
  \int_I \cdots \int_I
  \prod_{j=1}^m \frac{z_j z'_j}{z_j + z'_j}
  \varphi(\xi_j)\varphi(\xi'_j) \,d\xi_j d\xi'_j.
\end{equation}

It remains to estimate the integral
\[
\int_I \int_I \frac{z_j z'_j}{z_j + z'_j}
\varphi(\xi_j)\varphi(\xi'_j) \,d\xi_j d\xi'_j
=
\frac{1}{\log R} \int_I \int_I
\frac{(1 + i\xi_j)(1 + i\xi'_j)}{2 + i(\xi_j + \xi'_j)}
\varphi(\xi_j)\varphi(\xi'_j) \,d\xi_j d\xi'_j.
\]
We can replace the domain of integration $I = [-\log^{1/2} R,
\log^{1/2} R]$ by $\RR$ with a loss
of $O_A(\log^{-A} R)$ for any $A > 0$ due to the rapid decay
of $\varphi$ given by $\varphi(\xi) = O_A((1+\sabs{\xi})^{-A})$.
We claim
\begin{equation} \label{eq:c_chi-eval}
\int_\RR \int_\RR  \frac{(1 + i \xi)(1 + i \xi')}{2 + i(\xi + \xi')}
\varphi(\xi) \varphi(\xi') \, d\xi d\xi'
= \int_0^\infty \sabs{\chi'(x)}^2 \, dx
= c_\chi.
\end{equation}
Using
\[
\frac{1}{2 + i(\xi + \xi')} = \int_0^\infty e^{-(1+i\xi)x} e^{-(1 +
  i\xi')x} \, dx,
\]
we can rewrite the left-hand side of \eqref{eq:c_chi-eval} as
\[
\int_0^\infty\paren{\int_\RR \varphi(\xi) (1 + i\xi)e^{-(1+i\xi) x}
  \,d\xi}^2 \,dx.
\]
The expression in parentheses is $-\chi'(x)$, so \eqref{eq:c_chi-eval}
follows. Substituting \eqref{eq:c_chi-eval} into \eqref{eq:pre-c_chi-eval} we
arrive at the desired conclusion, Proposition~\ref{prop:Lambda-lf}.

\subsection*{Error estimates}

Now we bound the error terms in the above analysis.

\textit{Simple pole of Riemann zeta function.} Here is the argument
showing that $\zeta(s) = (s-1)^{-1} + O(1)$ whenever $\Re s > 1$ and
$s - 1 = O(1)$. We have $(s-1)^{-1} = \int_1^\infty x^{-s} \, dx$. So
\[
\zeta(s) - \frac{1}{s-1} = \sum_{n=1}^\infty n^{-s} - \int_1^\infty
x^{-s} \, dx = \sum_{n = 1}^\infty
  \int_{n}^{n+1} (n^{-s} - x^{-s}) \, dx.
\]
The $n$-th term on the right is bounded in magnitude by $O(n^{-2})$. So the sum is $O(1)$.

\medskip

\textit{Estimate \eqref{eq:post-swap}.} We want to bound the
difference between \eqref{eq:post-swap} and
\eqref{eq:pre-Fourier}. This means bounding the contribution to~\eqref{eq:pre-Fourier} from the
error term in
\eqref{eq:int-trunc-prod}. Taking absolute values everywhere, we bound
these contributions by
\begin{align*}
  &\ll_A (\log R)^{O(1) - A} \sum_{\substack{d_1, d'_1, \dots, d_m, d'_m
    \\ \text{sq-free integers}}}  \EE_{\bx \in \ZZ_D^t}[1_{d_j,d'_j
      \mid \theta_j(\bx) \ \forall j}] (d_1d'_1\cdots d_md'_m)^{-1/\log R}
  \\
  &= (\log R)^{O(1) - A} \prod_p \sum_{d_1, d'_1, \dots,
      d_m, d'_m \in \{1,p\}} \EE_{\bx \in \ZZ_p^t} [1_{d_j,d'_j
      \mid \theta_j(\bx) \ \forall j}]  (d_1d'_1\cdots d_md'_m)^{-1/\log R}.
\end{align*}
The expectation $\EE_{\bx \in \ZZ_p^t}[1_{d_j,d'_j \mid \theta_j(\bx) \ \forall
    j}] $ is $1$ if all $d_i$ and $d'_i$ are 1 and at most $1/p$
otherwise. We continue to bound the above by
\begin{align*}
  &\leq (\log R)^{O(1) - A} \prod_p \paren{1 + p^{-1}
    \sum_{\substack{d_1, d'_1, \dots, d_m, d'_m \in \{1,p\} \\
        \text{not all 1's}}} (d_1d'_1\cdots d_md'_m)^{-1/\log R}}
    \nonumber
    \\
    &= (\log R)^{O(1) - A} \prod_p \paren{1 + p^{-1}((p^{-1/\log R}
      +1)^{2m} - 1)} \nonumber
    \\
    &\leq (\log R)^{O(1) - A} \prod_p \paren{1 - p^{-1- 1/\log
        R}}^{-O(1)} \nonumber
    \\
    &= (\log R)^{O(1) - A} \zeta(1 + 1/\log
    R)^{O(1)}. \label{eq:truncate-error}
\end{align*}
So the difference between \eqref{eq:post-swap} and
\eqref{eq:pre-Fourier} is $O_A((\log R)^{O(1) - A})$, which
is small as long as we take $A$ to be sufficiently large.

\medskip

\textit{Estimate in \eqref{E'_p-zeta}.}
We have $\sabs{z_j}, \sabs{z'_j} = O(\log^{-1/2} R)$ since
$\sabs{\xi_j}, \sabs{\xi'_j} \leq \log^{1/2} R$. So
\begin{equation} \label{eq:E'_p-prod-est}
 \prod_{j=1}^m \frac{\zeta(1 + z_j +
    z'_j)}{\zeta(1 + z_j)\zeta(1 + z'_j)}
  = \prod_{j=1}^m \frac{((z_j + z'_j)^{-1} + O(1))}{(z_j^{-1} +
    O(1))({z'_j}^{-1} + O(1))} \\
  = (1 + O(\log^{-1/2} R))\prod_{j=1}^m \frac{z_j z'_j}{z_j + z'_j}.
\end{equation}

\medskip

\textit{Estimate in \eqref{eq:E'_p-phi}.}
If $\abs{z}\log p = O(1)$ (which is the case for $p \leq w$), then
\[
1 - p^{-1-z} = 1-p^{-1} e^{-z\log p} = 1 - p^{-1}(1 + O(\sabs{z} \log
p))  = (1 - p^{-1})(1 + O(\sabs{z} p^{-1} \log p)).
\]
It follows that for all $p \leq w$ and $\xi_1, \xi'_1, \dots, \xi_m, \xi'_m \in I$, we
have
\[
E'_p(\xi) = \paren{1 + O\paren{\frac{\log p}{p \log^{1/2}
      R}}}(1 - p^{-1})^m
\]
and, hence,
\begin{equation}\label{eq:E'_p-small}
\prod_{p \leq w} E'_p(\xi) = \paren{1 + O\paren{\frac{w}{\log^{1/2} R}}} \prod_{p \leq w} (1 - p^{-1})^m .
\end{equation}

\medskip

\textit{Estimate in \eqref{eq:pre-c_chi-eval}.} Using \eqref{eq:euler-compare},
\eqref{eq:E'_p-prod-est}, and \eqref{eq:E'_p-small}, we find that the
ratio between the two sides in \eqref{eq:pre-c_chi-eval} is
$1 + O(1/w + w/\log^{1/2} R) = 1 + o(1)$, as long as $w$ grows sufficiently slowly.

\section{Extensions of the Green-Tao theorem} \label{sec:conclusion}

We conclude by discussing a few extensions of the Green-Tao theorem.

\subsection*{Szemer\'edi's theorem in the primes}

As noted already by Green and Tao \cite{GT08}, their method also implies a Szemer\'edi-type theorem for the primes. That is, every subset of the primes with positive relative upper density contains arbitrarily long arithmetic progressions.

One elegant corollary of this result is that there are arbitrarily long APs where
every term is a sum of two squares. This result follows from a combination of the
well-known fact that every prime of the form $4n+1$ is a sum of two
squares with Dirichlet's theorem on primes in arithmetic progressions, which tells us that roughly half the primes are
congruent to $1$ (mod $4$). Even this innocent-sounding corollary was open before Green and Tao's paper.

\subsection*{Gaussian primes contain arbitrarily shaped constellations}

The Gaussian integers is the set of all numbers of the form $a + bi$, where $a, b \in \ZZ$. This set is a ring under the usual definitions of addition and multiplication for complex numbers. It is also a unique factorization domain, so it is legitimate to talk about the set of Gaussian primes. Tao~\cite{Tao06jam} proved that an analogue of the Green-Tao theorem holds for the Gaussian primes.


We say that $A \subseteq \ZZ^d$ \emph{contains arbitrary constellations} if, for every finite set $F \subseteq \ZZ^d$, there
exist $x \in \ZZ^d$ and $t \in \ZZ_{>0}$ such that $x + tf
\in A$ for every $f \in F$. Tao's theorem then states that the Gaussian primes, viewed as a subset of $\ZZ^2$, contain arbitrary constellations. Just as the Green-Tao theorem uses Szemer\'edi's theorem as a black box, this theorem uses the multidimensional analogue of Szemer\'edi's theorem, first proved by Furstenberg and Katznelson~\cite{FK78}. This states that every subset of $\ZZ^d$ with positive upper density\footnote{A set $A \subseteq \ZZ^d$ is said to have positive
  upper density if $\limsup_{N \to
    \infty} \abs{A \cap [-N,N]^d}/{(2N+1)^d} > 0$. We say that $A
  \subseteq S \subseteq \ZZ^d$ has
  positive relative upper density if $\limsup_{N \to
    \infty} \abs{A \cap S \cap [-N,N]^d}/\abs{S \cap [-N,N]^d} > 0$.} contains arbitrary constellations.
The Furstenberg-Katznelson theorem also follows from the hypergraph removal lemma and the approach taken by
Tao is to transfer this hypergraph removal proof to the sparse context. It may therefore be
seen as a precursor to the approach taken here.

\subsection*{Multidimensional Szemer\'edi theorem in the primes}

Let $P$ denote the set of primes in $\ZZ$. It was shown recently by
Tao and Ziegler \cite{TZ13ar} and, independently, by Cook, Magyar, and
Titichetrakun \cite{CMT13ar}, that every subset of $P^d$ of positive
relative upper density contains arbitrary constellations. A
short proof was subsequently given in \cite{FZ13ar} (though, like
\cite{TZ13ar}, it assumes some difficult results of Green, Tao, and Ziegler that
we will discuss later in this section).

Although both this result and Tao's result on the Gaussian primes are multidimensional
analogues of the Green-Tao theorem, they are quite different in
nature. Informally speaking, a key difficulty in the second result is that there is
a strong correlation between coordinates in $P^d$ (namely, that all coordinates
are simultaneously prime), whereas there is no significant correlation
between the real and imaginary parts of a typical Gaussian prime
(after applying an extension of the $W$-trick).

\subsection*{The primes contain arbitrary polynomial
progressions}

We say that $A \subseteq \ZZ$ \emph{contains arbitrary polynomial
  progressions} if, whenever $P_1, \dots, P_k \in \ZZ[X]$ are
polynomials in one variable with integer coefficients satisfying
$P_1(0) = \cdots = P_k(0) = 0$, there is some $x \in \ZZ$ and $t \in
\ZZ_{> 0}$ such that $x + P_j(t) \in A$ for each $j = 1, \dots, k$.
A striking generalization of Szemer\'edi's theorem due to Bergelson
and Leibman~\cite{BL96} states that any subset of $\ZZ$ of positive
upper density contains arbitrary polynomial progressions. To date, the only
known proofs of this result use ergodic theory.

For primes, an analogue of the Bergelson-Leibman theorem was proved
by Tao and Ziegler~\cite{TZ08}. This result states that any subset of the primes with
positive relative upper density contains arbitrary polynomial
progressions. In particular, the primes themselves contain arbitrary polynomial
progressions. It seems plausible that the simplifications outlined here could also
be used to simplify the proof of this theorem.

\subsection*{The number of $k$-APs in the primes}

The original approach of Green and Tao (and the approach outlined in this paper) implies that for any $k$
the number of $k$-APs of primes with each term at most $N$ is on the order of $\frac{N^2}{\log^k N}$. In subsequent work, Green, Tao, and Ziegler \cite{GT10,GT12aofm,GTZ12aofm} showed how to determine the exact asymptotic. That is, they determine a constant $c_k$ such that the number of $k$-APs of primes with each term at most $N$ is
$(c_k + o(1)) \frac{N^2}{\log^k N}$. More generally, they determine an asymptotic for the number of prime solutions to a broad range of linear systems of equations.

The proof of these results also draws on the transference technique discussed in this paper but a number of additional ingredients are needed, most notably an inverse theorem describing the structure of those sets which do not contain the expected number of solutions to certain linear systems of equations. It is this result which is transferred to the sparse setting when one wishes to determine the exact asymptotic.

\bigskip

\noindent\textbf{Acknowledgments.} We thank Yuval Filmus, Mohammad
Bavarian, and the anonymous referee for helpful comments on the manuscript.


\end{document}